\documentclass{amsart}
\usepackage{amsfonts,amssymb,amsmath,amsthm}
\usepackage[colorlinks,citecolor=blue,urlcolor=black,linkcolor=black]{hyperref}
\usepackage{enumerate}
\usepackage[headings]{fullpage}
\usepackage{cmap}

\usepackage{mathrsfs}
\usepackage{bbold}

\newtheorem*{theorema}{Theorem A}
\newtheorem*{theoremb1}{Theorem B1}
\newtheorem*{theoremb2}{Theorem B2}
\newtheorem*{theoremc1}{Theorem C1}
\newtheorem*{theoremc2}{Theorem C2}
\newtheorem*{theoremc3}{Theorem C3}
\newtheorem*{theoremd}{Theorem D}
\newtheorem{theorem}{Theorem}[section]
\newtheorem{claim}{Claim}[theorem]
\newtheorem{lemma}[theorem]{Lemma}
\newtheorem{fact}[theorem]{Fact}
\newtheorem{prop}[theorem]{Proposition}
\newtheorem{corollary}[theorem]{Corollary}
\theoremstyle{definition}
\newtheorem{definition}[theorem]{Definition}
\newtheorem*{defn}{Definition}
\numberwithin{equation}{section}
\theoremstyle{remark}
\newtheorem*{remarks}{Remarks}

\newcommand{\zfc}{\mathnormal{\mathsf{ZFC}}}
\newcommand{\ch}{\mathnormal{\mathsf{CH}}}
\newcommand{\gch}{\mathnormal{\mathsf{GCH}}}
\newcommand{\fs}{\mathnormal{\mathrm{FS}}}
\newcommand{\fu}{\mathnormal{\mathrm{FU}}}
\newcommand{\sumsets}{\mathnormal{\mathrm{SuS}}}
\newcommand{\pair}[1]{\langle#1\rangle}
\def\s{\subseteq}

\def\symdiff{\mathbin\bigtriangleup}
\def\Mid{\mathrel{\bigg|}}
\def\br{\blacktriangleright}
\def\orderunion{\sqcup}
\DeclareMathOperator{\ord}{ord}
\DeclareMathOperator{\pr}{Pr}
\DeclareMathOperator{\pp}{pp}
\DeclareMathOperator{\im}{Im}
\DeclareMathOperator{\osc}{osc}
\DeclareMathOperator{\axiom}{S}
\DeclareMathOperator{\axiomi}{S^*}
\DeclareMathOperator{\supp}{supp}
\DeclareMathOperator{\cf}{cf}

\DeclareMathOperator{\otp}{otp}

\author[D. Fern\'andez]{David Fern\'andez-Bret\'on}
\address{
Department of Mathematics\\
University of Michigan\\
2074 East Hall, 530 Church Street \\
Ann Arbor, MI 48109-1043, U.S.A.}
\urladdr{http://www-personal.umich.edu/\textasciitilde djfernan/}
\author[A. Rinot]{Assaf Rinot}
\address{Department of Mathematics, Bar-Ilan University, Ramat-Gan 52900, Israel.}
\urladdr{http://www.assafrinot.com}

\thanks{The first author was partially supported by Postdoctoral Fellowship
number 263820/275049 from the Consejo Nacional de Ciencia y Tecnolog\'{\i}a
(CONACyT), Mexico.
The second author was partially supported by the Israel Science Foundation (grant $\#$1630/14).}

\begin{document}

\title[Strong failures of Hindman's Theorem]{Strong failures of higher analogs of Hindman's Theorem}

\begin{abstract}
We show that various analogs of Hindman's Theorem fail in a strong sense
when one attempts to obtain uncountable monochromatic sets:

{\bf Theorem 1.} There exists a colouring $c:\mathbb R\rightarrow\mathbb Q$, such that
for every $X\subseteq\mathbb R$ with $|X|=|\mathbb R|$, and every colour $\gamma\in\mathbb Q$,
there are two distinct elements $x_0,x_1$ of $X$ for which $c(x_0+x_1)=\gamma$.
This forms a simultaneous generalization of a theorem of Hindman, Leader and Strauss and a theorem of Galvin and Shelah.

{\bf Theorem 2.} For every Abelian group $G$,
there exists a colouring $c:G\rightarrow\mathbb Q$ such that for every uncountable $X\subseteq G$, and every colour $\gamma$,
for some large enough integer $n$,
there are pairwise distinct elements $x_0,\ldots,x_n$ of $X$ such that $c(x_0+\cdots+x_n)=\gamma$.
In addition, it is consistent that the preceding statement remains valid even after enlarging the set of colours from $\mathbb Q$ to $\mathbb R$.

{\bf Theorem 3.} Let $\circledast_\kappa$ assert that
for every Abelian group $G$ of cardinality $\kappa$,
there exists a colouring $c:G\rightarrow G$
such that for every positive integer $n$, every  $X_0,\ldots,X_n \in[G]^\kappa$, and every $\gamma\in G$,
there are $x_0\in X_0,\ldots, x_n\in X_n$ such that $c(x_0+\cdots+x_n)=\gamma$.
Then $\circledast_\kappa$ holds for unboundedly many uncountable cardinals $\kappa$,
and it is consistent that $\circledast_\kappa$ holds for \emph{all} regular uncountable cardinals $\kappa$.

\end{abstract}
\dedicatory{This paper is dedicated to the memory of Andr\'as Hajnal (1931--2016)}
\date{\today}
\maketitle

\section{Introduction}

In one of its more general forms, Hindman's
Theorem (see~\cite[Corollary 5.9]{MR2893605} for the
general form; the particular case $G=\mathbb N$ was originally proved in~\cite{MR0349574})
asserts that whenever a commutative cancellative semigroup $G$ is
partitioned into two
cells (i.e., coloured with two colours), there exists an infinite
$X\subseteq G$ such that the set of its \textit{finite sums}
\begin{equation*}
\fs(X):=\left\{\sum_{x\in a}x\Mid a\in[X]^{­<\omega}\right\}
\end{equation*}
is completely contained in one of the cells of the
partition (i.e., $\fs(X)$ is \emph{monochromatic}). The infinite set $X\subseteq G$
constructed in the proof of this theorem is countable,
so it is natural to ask whether it is possible to find,
given a colouring of an uncountable commutative cancellative
semigroup $G$, a
subset $X\subseteq G$ of a given uncountable cardinality
such that
$\fs(X)$ is monochromatic. This question was answered
in the negative in~\cite{fb-on-hindman}, where, given
a commutative cancellative semigroup $G$, a colouring with
two colours of
$G$ is exhibited such that no uncountable $X\subseteq G$
can have $\fs(X)$ monochromatic. A related result for the
particular case of the group
$\mathbb R$ can be
found in~\cite{hindman-leader-strauss}, where
a colouring of $\mathbb R$ with two colours is exhibited,
satisfying that whenever $X\subseteq\mathbb R$ has the same cardinality
as $\mathbb R$, then not only is $\fs(X)$ not monochromatic,
but even $\fs_2(X):=\{x+y\mid x,y\in X\text{ distinct}\}$
is not monochromatic. In particular, assuming the Continuum
Hypothesis ($\ch$),\footnote{$\ch$ asserts that $|\mathbb R|=\omega_1$, that is, that there
do not exist any cardinalities strictly between that of
$\mathbb N$ and that of $\mathbb R$.} this result implies that
for the aforementioned colouring of $\mathbb R$, every uncountable
subset $X$ is such that $\fs_2(X)$ is not monochromatic.

In this paper we shall consider stronger versions
of these results, where colourings of uncountable commutative cancellative
semigroups $G$ are obtained, with more than two colours,
satisfying that for every uncountable $X$ not only is $\fs(X)$ not
monochromatic, but in fact $\fs(X)$ contains occurrences of
every possible colour. In order to properly state our results,
we will introduce some new notation, which is inspired by the
standard notation from Ramsey theory. Following \cite[p.~56]{MR795592},
for cardinals
$\kappa,\lambda,\theta,\mu$, write
\begin{equation*}
\kappa\rightarrow[\lambda]_\theta^\mu
\end{equation*}
to assert that for every colouring of $[\kappa]^\mu$ in $\theta$ many
colours, it is possible to find an $X\subseteq\kappa$ with
$|X|=\lambda$ such that $[X]^\mu$ omits at least one of the
colours.\footnote{Here, $[X]^\mu$ stands for the collection of all subsets of $X$ of cardinality $\mu$,
where in the case that $X$ is a set of ordinals,
we identify these $\mu$-sized subsets with their increasing enumeration.
In particular, $[\kappa]^2=\{\pair{\alpha,\beta}\mid \alpha<\beta<\kappa\}$.} Note that the negation of the above, denoted
\begin{equation*}
\kappa\nrightarrow[\lambda]_\theta^\mu,
\end{equation*}
asserts the existence of a colouring $c:[\kappa]^\mu\rightarrow\theta$ such that for every $\lambda$-sized subset $X$ of $\kappa$,
we have $c``[X]^\mu=\theta$.
So Ramsey's theorem is just the assertion that
$\omega\rightarrow[\omega]_2^2$ holds, however, when $\omega$ is replaced
by larger cardinals, typically one gets negative relations,
sometimes quite strong (i.e., involving a large number of colours).
Negative square bracket partition relations have been studied extensively.
For instance, in the realm of $\omega_1$, we have a sequence of results
starting with Sierpi\'nski's uncountable poset \cite{MR1556708} that admits no uncountable chains nor uncountable antichains, thereby, witnessing $\omega_1\nrightarrow[\omega_1]_2^2$.
Later, Blass \cite{MR0297576} improved this to $\omega_1\nrightarrow[\omega_1]_3^2$,
and  Galvin and Shelah \cite{MR0329900} improved further to $\omega_1\nrightarrow[\omega_1]_4^2$.
Then, Todor\v{c}evi\'c \cite{MR908147} managed to gain control on the maximal number of colours, proving that $\omega_1\nrightarrow[\omega_1]_{\omega_1}^2$ holds.
Recently, even more complicated statements were proven by Moore \cite{MR2220104} and Peng and Wu \cite{pengwu}.

In analogy with the above, we now define a negative partition relation for commutative semigroups, involving finite sums ($\fs$),
bounded finite sums $(\fs_n$), and sumsets ($\sumsets$).
\begin{defn} For a commutative semigroup $G$, cardinals $\lambda,\theta$, and an integer $n$:
\begin{enumerate}
\item $G\nrightarrow[\lambda]^\fs_\theta$ asserts the existence of a colouring $c:G\rightarrow\theta$
such that for every $\lambda$-sized subset $X$ of $G$, we have $c[\fs(X)]=\theta$;
\item $G\nrightarrow[\lambda]^{\fs_n}_\theta$  asserts the existence of a colouring $c:G\rightarrow\theta$
such that for every $\lambda$-sized subset $X$ of $G$, we have $c[\fs_n(X)]=\theta$, where $\fs_n(X):=\{x_1+\cdots+x_n\mid x_1,\ldots,x_n\in X\text{ are all distinct}\}$;
\item $G\nrightarrow[\lambda]^\sumsets_\theta$ asserts the existence of a colouring $c:G\rightarrow\theta$
such that for every integer $m\ge2$, and every $\lambda$-sized subsets $X_1,\ldots,X_m$ of $G$, we have $c[X_1+\cdots+X_m]=\theta$,
where $X_1+\cdots+X_m:=\{x_1+\cdots+x_m\mid \forall i(x_i\in X_i)\}$.
\end{enumerate}
\end{defn}

Note that for all $\lambda\le\lambda'$, $G\s G'$, $\theta\le\theta'$,
and $\texttt{x}\in\{\fs_n,\fs,\sumsets\mid n<\omega\}$,
$G'\nrightarrow[\lambda]_{\theta'}^\texttt{x}$ implies  $G\nrightarrow[\lambda']_\theta^\texttt{x}$.
Also note that for every $n$, $G\nrightarrow[\lambda]^{\fs_n}_\theta$ implies $G\nrightarrow[\lambda]^{\fs}_\theta$.
Finally, note that for every infinite cardinal $\lambda$ and every integer $n\ge 2$, $G\nrightarrow[\lambda]^\sumsets_\theta$ implies $G\nrightarrow[\lambda]^{\fs_n}_\theta$
simply because any infinite set $X$ may be partitioned into $\biguplus_{i=1}^nX_i$ in such a way that $|X_i|=|X|$ for all $i$.
And indeed, $G\nrightarrow[\lambda]^\sumsets_\theta$ will be the strongest negative partition relation considered in this paper.

\medskip

Using the above notation, we can now succinctly state the relevant web of results (in chronological order):
\begin{itemize}
\item The generalized Hindman theorem (see \cite[Corollary 5.9]{MR2893605}) asserts that $G\rightarrow[\omega]_2^\fs$  holds (that is, $G\nrightarrow[\omega]_2^\fs$ fails)
for every infinite commutative cancellative semigroup $G$;\footnote{Consequently, $G\rightarrow[\omega]_n^\fs$ holds for every $n<\omega$.}
\item Milliken \cite{MR0505558} proved that $G\nrightarrow[\kappa^+]^{\fs_2}_{\kappa^+}$ holds whenever $|G|=\kappa^+=2^\kappa$ for some infinite cardinal $\kappa$;\footnote{However, it is consistent with $\zfc$ that $2^\kappa>\kappa^+$ for every infinite cardinal $\kappa$ \cite{MR1087344}.}
\item Hindman, Leader and Strauss \cite[Theorem~3.2]{hindman-leader-strauss} proved
that $\mathbb R\nrightarrow[\mathfrak c]_2^{\fs_n}$ holds for every integer $n\ge2$;
\item The first author~\cite{fb-on-hindman} proved that
$G\nrightarrow[\omega_1]_2^\fs$ holds for every uncountable commutative
cancellative semigroup $G$;
\item Komj\'ath~\cite{komjath}, and independently Soukup and Weiss~\cite{dani-bill}, proved that $\mathbb R\nrightarrow[\mathfrak\omega_1]_2^{\fs_n}$ holds for every integer $n\ge2$.
They also pointed out \cite[Corollary~2.3]{dani-bill} that by a theorem of Shelah \cite[Theorem~2.1]{MR955139},
it is consistent with $\zfc$ (modulo a large cardinal hypothesis) that $\mathbb R\nrightarrow[\omega_1]_3^{\fs_n}$ fails for every $n\ge2$.
\end{itemize}

\medskip

The main results of this paper read as follows:
\begin{theorema} $G\nrightarrow[\omega_1]_\omega^\fs$ holds for every uncountable commutative cancellative semigroup $G$.
\end{theorema}

We shall show that the superscript $\fs$ in the preceding is optimal, by exhibiting an uncountable abelian group $G$ for which $G\nrightarrow[\omega_1]^{\fs_n}_\omega$ fails for all $n$.
We shall also address a stronger form of Theorem A, proving that things can go both ways:
\begin{theoremb1} It is consistent with $\zfc$ that $G\nrightarrow[\omega_1]_{\omega_1}^\fs$ holds for every uncountable commutative cancellative semigroup $G$.
\end{theoremb1}

\begin{theoremb2} Modulo a large cardinal hypothesis,
it is consistent with $\zfc$ that $\mathbb R\nrightarrow[\omega_1]^{\fs}_{\omega_1}$ fails.
\end{theoremb2}

It turns out that partition relations for $\fs_n$ sets can also go both ways.
To exemplify on the real line:

\begin{theoremc1}
If $\mathfrak c$ is a successor cardinal (e.g., assuming $\ch$), then $\mathbb R\nrightarrow[\mathfrak c]^{\sumsets}_{\omega_1}$ holds,
and hence, so does $\mathbb R\nrightarrow[\mathfrak c]^{\fs_n}_{\omega_1}$ for every integer $n\ge2$.
\end{theoremc1}

\begin{theoremc2} Modulo a large cardinal hypothesis,
it is consistent with $\zfc$ that $\mathbb R\nrightarrow[\mathfrak c]^{\fs_n}_{\omega_1}$ fails for every integer $n\ge2$.
\end{theoremc2}

The preceding raises the question of which negative partition relations for $\mathbb R$ are consequences of $\zfc$.
For this, we have the following simultaneous generalization of \cite[Theorem~1]{MR0329900} and \cite[Theorem~3.2]{hindman-leader-strauss}:

\begin{theoremc3} $\mathbb R\nrightarrow[\mathfrak c]_\omega^{\fs_n}$ holds for every integer $n\ge 2$.
\end{theoremc3}

Finally, we establish that the negative partition relation of the strongest form is quite a prevalent phenomena:

\begin{theoremd} Denote by $\circledast_\kappa$
the assertion that  $G\nrightarrow[\kappa]_{\kappa}^{\sumsets}$ holds for every commutative cancellative semigroup $G$ of cardinality $\kappa$.
Then:
\begin{itemize}
\item $\circledast_\kappa$ holds for $\kappa=\aleph_1,\aleph_2,\ldots,\aleph_n,\ldots$;
In fact, $\circledast_\kappa$ holds for every $\kappa$ which is a successor of regular cardinal;

\item $\circledast_{\kappa}$ holds whenever $\kappa=\lambda^+=2^\lambda$
or whenever $\kappa$ is a regular uncountable cardinal admitting a nonreflecting stationary set. In particular:
\item It is consistent with $\zfc$ that $\circledast_\kappa$
holds for every regular uncountable cardinal $\kappa$.
\end{itemize}
\end{theoremd}

\subsection*{Organization of this paper}
Theorems A, B1 and B2 are proved in Section~\ref{omegacolours}.
In Section~\ref{infcombinatorics}, we establish some new results on the partition calculus of uncountable cardinals.
In Section~\ref{negativehindman}, the machinery of Section~\ref{infcombinatorics} is invoked in proving, among other things, Theorem D.
Section~\ref{therealline} focuses on  the real line, and Theorems C1, C2, C3 are derived there as corollaries.
\section{Colourings for finite sums}\label{omegacolours}

We open this section by stating a structural result that will
allow us to pass from elements of commutative cancellative semigroups of
cardinality $\kappa$ to finite subsets of $\kappa$, so that we
are able to apply some machinery from partition relations on
cardinals to our semigroups.\footnote{The noncommutative case will be handled in a forthcoming paper.}
For this, we will need to lay down some terminology.

\begin{definition}
Given a sequence of groups $\langle G_\alpha\mid\alpha<\kappa\rangle$,
define its \emph{direct sum} to be the group
\begin{equation*}
\bigoplus_{\alpha<\kappa}G_\alpha:=\left\{x\in\prod_{\alpha<\kappa} G_\alpha \Mid x(\alpha)\text{ equals the identity for all but finitely many }\alpha\right\}.
\end{equation*}
\end{definition}

Recall that a \emph{divisible group} is an abelian group $G$
such that for every $x\in G$ and every $n\in\mathbb N$, there exist
some $z\in G$ such that $nz=x$.

\begin{lemma}\label{embedding}
Suppose that $G$ is an infinite commutative cancellative semigroup.
Denote $\kappa:=|G|$. Then there exists a sequence of countable divisible groups, $\langle G_\alpha\mid \alpha<\kappa\rangle$,
such that $G$ embeds in $\bigoplus_{\alpha<\kappa}G_\alpha$.
\end{lemma}
\begin{proof}
It is well-known that every commutative cancellative semigroup
$G$ can be embedded in an abelian group $G'$, which can furthermore be assumed to
have the same cardinality as $G$.\footnote{This is
done by means of the same process which embeds the additive
group $\mathbb N$ into $\mathbb Z$, or the multiplicative group
$\mathbb Z\setminus\{0\}$ into $\mathbb Q\setminus\{0\}$. This
process yields $G'$ as a quotient of the semigroup $G\times G$,
therefore $|G'|=|G|$.}
Next, since $G'$ is an abelian group, it can be embedded in a divisible
group $G''$, by~\cite[Theorem 24.1]{MR0255673}.
Finally, by~\cite[Theorem 23.1]{MR0255673}, each divisible group is isomorphic to a
direct sum of some copies of $\mathbb Q$ with some
quasicyclic groups $\mathbb Z(p^\infty)$.\footnote{Recall that for a prime number $p$, the \emph{$p$-quasicyclic group}
(also known as the \emph{Pr\"ufer $p$-group})
is a countable divisible subgroup of $\mathbb R/\mathbb Z$, defined by
$\mathbb Z(p^\infty):=\left\{\frac{a}{p^n}+\mathbb Z\Mid a\in\mathbb Z\ \&\ n<\omega\right\}.$}
Thus, we may assume that
$G''=\bigoplus_{\alpha<\lambda}G_\alpha$, where
each $G_\alpha$ is equal to either $\mathbb Q$, or
$\mathbb Z(p^\infty)$. By removing spurious summands
(i.e., the $\alpha<\lambda$ such that $x(\alpha)=0$
for every $x\in G$) we can assume that the number of
summands is $|G|$, in other words, that $\lambda=\kappa$.
Thus, $G$ embeds into $\bigoplus_{\alpha<\kappa}G_\alpha$,
and each $G_\alpha$ (being either $\mathbb Q$ or
$\mathbb Z(p^\infty)$ for some $p$) is a countable divisible group.
\end{proof}

In what follows, given an infinite commutative cancellative
semigroup $G$ of infinite cardinality $\kappa$, we will
always implicitly fix an embedding of
$G$ into $\bigoplus_{\alpha<\kappa}G_\alpha$, where
each $G_\alpha$ is countable, as per Lemma~\ref{embedding}.
This allows us to define the \emph{support} of an element $x\in G$ to be
the finite set
\begin{equation*}
\supp(x):=\{\alpha<\kappa\mid x(\alpha)\neq 0\}.
\end{equation*}

\begin{definition} A family of sets $\mathcal X$ is said to be a \emph{$\Delta$-system with root $r$}
if for every two distinct $x,x'\in\mathcal X$, we have $x\cap x'=r$.

A $\Delta$-system $\mathcal X$ is said to be of the \emph{head-tail-tail} form if:
\begin{itemize}
\item $\sup(r)<\min(x\setminus r)$ for all $x\in\mathcal X$;
\item for any two distinct $x,x'\in\mathcal X$, either $\sup(x)<\min(x'\setminus r)$ or $\sup(x')<\min(x\setminus r)$.
\end{itemize}
\end{definition}

A standard fact from Set Theory states that for every regular uncountable cardinal $\kappa$ and every family $\mathcal X$
consisting of $\kappa$ many finite sets, there exists $\mathcal X'\s\mathcal X$ with $|\mathcal X'|=\kappa$
such that $\mathcal X'$ forms a $\Delta$-system. In the special case that $\mathcal X\s[\kappa]^{<\omega}$,
a $\Delta$-subsystem $\mathcal X'\s\mathcal X$ may be found which is moreover of the head-tail-tail form.

\begin{prop}\label{supportofsums}
Suppose that $x_1,\ldots,x_n$ are elements
of a direct sum $\bigoplus_{\alpha<\kappa}G_\alpha$, and there
exist a fixed $r\in[\kappa]^{<\omega}$
and pairwise disjoint $s_1,\ldots,s_n\in[\kappa]^{­<\omega}$
satisfying $\supp(x_i)=r\uplus s_i$ for all $1\leq i\leq n$
(that is, the set of corresponding supports forms a $\Delta$-system
with root $r$). Then
\begin{equation*}
s_1\cup\cdots\cup s_n\subseteq\supp(x_1+\cdots+x_n)\subseteq r\cup s_1\cup\cdots\cup s_n.
\end{equation*}
\end{prop}

\begin{proof}
Let $\alpha<\kappa$ be arbitrary. If $\alpha\notin r\cup s_1\cup\cdots\cup s_n$,
then $\alpha\notin\supp(x_i)$ for any $i$, thus
\begin{equation*}
(x_1+\cdots+x_n)(\alpha)=x_1(\alpha)+\cdots+x_n(\alpha)=0+\cdots+0=0,
\end{equation*}
therefore $\supp(x_1+\cdots+x_n)\subseteq r\cup s_1\cup\cdots\cup s_n$.
Now, if $\alpha\in s_1\cup\cdots\cup s_n$, then there exists
a unique $1\leq i\leq n$ with $\alpha\in s_i$. This
means that $\alpha\in\supp(x_i)$ but $\alpha\notin\supp(x_j)$
for $j\neq i$, in other words, $x_i(\alpha)\neq 0$ but
$x_j(\alpha)=0$ for $j\neq i$. Therefore
\begin{eqnarray*}
(x_1+\cdots+x_n)(\alpha) & = & x_1(\alpha)+\cdots+x_{i-1}(\alpha)+x_i(\alpha)+x_{i+1}(\alpha)+x_n(\alpha) \\
 & = & 0+\cdots+0+x_i(\alpha)+0+\cdots+0=x_i(\alpha)\neq 0,
\end{eqnarray*}
thus $s_1\cup\cdots\cup s_n\subseteq\supp(x_1+\cdots+x_n)$.
\end{proof}

The main offshot of Proposition~\ref{supportofsums} is that, whenever
we want to determine $\supp(x_1+\cdots+x_n)$ for $x_1,\ldots,x_n$
satisfying the corresponding hypothesis, the only coordinates $\alpha$
that require careful inspection are the $\alpha\in r$,
where one must determine whether
$x_1(\alpha)+\ldots+x_n(\alpha)$ is equal to $0$. In the
particular case where $n=2$, we obtain the fact, which will be
useful later, that
\begin{equation*}
\supp(x_1)\symdiff\supp(x_2)\subseteq\supp(x_1+x_2)\subseteq\supp(x_1)\cup\supp(x_2).
\end{equation*}

In~\cite[Theorem~5]{fb-on-hindman}, it was established that for every
uncountable commutative cancellative semigroup $G$,
$G\nrightarrow[\omega_1]_2^\fs$. This was done by colouring an
element $x\in G$ according to the parity of
$\lfloor\log_2|\supp(x)|\rfloor$.\footnote{The idea of colouring a finite
set $F$ according to the parity of $\lfloor\log_2|F|\rfloor$ goes
back to~\cite{MR906807}.} Similarly, for every uncountable
commutative cancellative semigroup $G$ and every positive integer $m$,
it is possible to define a colouring $c:G\longrightarrow m$ by declaring
$c(x)$ to be the class of $\lfloor\log_2|\supp(x)|\rfloor$ modulo $m$.
Arguing in the same way as in the proof
of~\cite[Theorem 5]{fb-on-hindman}, one can show that
every uncountable $X\subseteq G$  satisfies $c[\fs(X)]=m$; so that
every uncountable commutative cancellative semigroup $G$ satisfies
$G\nrightarrow[\omega_1]_m^\fs$
for every $m<\omega$. In this section, among other things, we will show that
every uncountable commutative cancellative semigroup $G$ moreover satisfies
$G\nrightarrow[\omega_1]_\omega^\fs$.

In order to simplify certain steps of the proof of the main theorem,
let us introduce the following notion.

\begin{definition}
Let $G$ be a commutative semigroup, and $X\subseteq G$.

We say that $Y$ is a \emph{condensation} of $X$
if there exists a family $\mathcal A\s[X]^{<\omega}$ consisting of pairwise disjoint sets, such that
\begin{equation*}
Y=\left\{\sum_{x\in a}x\Mid a\in\mathcal A\right\}.
\end{equation*}
\end{definition}

The importance of this notion stems from the evident fact that if $Y$ is a condensation of $X$,
then $\fs(Y)\subseteq\fs(X)$. Thus the notion of ``passing to a condensation"
constitutes another, purely algebraic, form of thinning-out a family of
elements of $G$ which can be useful if we are dealing with finite sums without
restrictions on the number of summands.

The following argument was first used
in~\cite{fb-on-hindman} (in the proof of the Claim within the proof of
Theorem~5 therein), and is encapsulated here as a lemma both for
the convenience of the reader, and for future reference.

\begin{lemma}\label{gotocondensation}
Suppose that $G$ is a commutative cancellative
semigroup, and  $X\subseteq G$ is a subset of
regular uncountable cardinality. Fix an embedding of
$G$ into a direct sum  $\bigoplus_{\alpha<\kappa}G_\alpha$, as per Lemma~\ref{embedding}.

Then there exists a condensation $Y$ of $X$ such that:
\begin{itemize}
\item $\supp[Y]:=\{\supp(y)\mid y\in Y\}$ forms a $\Delta$-system of cardinality $|X|$;
\item for all $y\in Y$ and all $\alpha$ in the root of  $\supp[Y]$,  $y(\alpha)$ has an infinite order in $G_\alpha$;
\item for every positive integer $n$ and every $y_1,\ldots,y_n\in Y$,
\begin{equation*}
\supp(y_1+\cdots+y_n)=\supp(y_1)\cup\cdots\cup\supp(y_n).
\end{equation*}
\end{itemize}
\end{lemma}
\begin{proof} Note that as
$x\mapsto\{\pair{\alpha,x(\alpha)}\mid\alpha\in\supp(x)\}$
is an injection, and each $G_\alpha$ is countable, we have
$|\supp[Y]|=|Y|$ for every uncountable $Y\s X$.
In particular, $|\supp[X]|$ is regular and uncountable,
and by passing to an equipotent subset of $X$, we may
assume that $\supp[X]$ forms a $\Delta$-system.

Let $r$ denote the root of this system. If $r$ is empty, then we are done,
as a consequence of Proposition~\ref{supportofsums}.
Thus, suppose that $r$ is nonempty.
Now, a finite number of applications
of the pigeonhole principle allows us to thin out $X$, without changing
its cardinality, in such a way that
for every $\alpha\in r$, there exists a fixed $g_\alpha\in G_\alpha$
such that $x(\alpha)=g_\alpha$ for all $x\in X$.
Let $r_\infty:=\{\alpha\in r\mid g_\alpha\text{ is of infinite order}\}$,
and  $M:=\{ \ord(g_\alpha)\mid\alpha\in r\setminus r_\infty\}$.
Since $r$ is finite, so is $M$, thus we can take $m:=\prod_{n\in M}n$ (in
the understanding that the empty product equals $1$), thereby ensuring
that $m g_\alpha=0$ for all
$\alpha\in r\setminus r_\infty$. Now we obtain a condensation $Y$ of $X$,
by fixing some family $\mathcal A\s[X]^m$ consisting of exactly $|X|$
many pairwise disjoint $m$-sized sets, and then letting
\begin{equation*}
Y:=\left\{\sum_{x\in a}x\Mid a\in\mathcal A\right\}.
\end{equation*}

\begin{claim}$\supp[Y]$ forms a $\Delta$-system of cardinality $|X|$
with root $r_\infty$.
\end{claim}
\begin{proof} For each $x\in X$, denote
$s_x:=\supp(x)\setminus r$, so that $\{ s_x\mid x\in X\}$ is a family of pairwise disjoint sets. Now, for any $y\in Y$, there is
an $a\in\mathcal A$ (with $|a|=m$) such that
$y=\sum_{x\in a}x$, and we claim that
\begin{equation*}
\supp(y)=r_\infty\uplus\left(\bigcup_{x\in a}s_x\right).
\end{equation*}

This will show that $\supp[Y]$ forms a $\Delta$-system with root $r_\infty$.
This will also show that the map $a\mapsto \supp(\sum_{x\in a}x)$ is an injection from $\mathcal A$ to $\supp[Y]$,
so that $|\supp[Y]|=|\mathcal A|=|X|$.

Here goes.
By Proposition~\ref{supportofsums}, we know that
$\bigcup_{x\in a}s_x\subseteq\supp(y)\subseteq r\cup\left(\bigcup_{x\in a}s_x\right)$,
thus it suffices to prove that for every $\alpha\in r$:
$\alpha\in\supp(y)\iff\alpha\in r_\infty$.

$\br$ For each $\alpha\in r_\infty$ and $x\in a$, we have $x(\alpha)=g_\alpha$
which is an element of infinite order. Hence
\begin{equation*}
y(\alpha)=\sum_{x\in a}x(\alpha)=m g_\alpha\neq 0.
\end{equation*}

$\br$ For each $\alpha\in r\setminus r_\infty$ and $x\in a$, we have $x(\alpha)=g_\alpha$
which is an element whose order is a divisor of $m$. Hence
\begin{equation*}
y(\alpha)=\sum_{x\in a}x(\alpha)=m g_\alpha=0.
\end{equation*}

Altogether, we have
$\supp(y)=r_\infty\uplus\left(\bigcup_{x\in a}s_x\right)$.

Finally, given any two distinct $y,y'\in Y$,
pick $a,a'\in\mathcal A$ such that $y=\sum_{x\in a}x$ and $y'=\sum_{x\in a'}x$.
Then by $a\cap a'=\varnothing$ and since the $s_x$ are pairwise disjoint, we have that
\begin{equation*}
\supp(y)\cap\supp(z)=\left(r_\infty\cup\bigcup_{x\in a}s_x\right)\cap\left(r_\infty\cup\bigcup_{x\in b}s_x\right)=r_\infty,
\end{equation*}
which shows that $\supp[Y]$ forms a $\Delta$-system, with root
$r_\infty$.
\end{proof}

Now let $n\in\mathbb N$, and let $y_1,\ldots,y_n\in Y$ be $n$ many
distinct elements. It remains to prove that
$\supp(y_1+\cdots+y_n)=\supp(y_1)\cup\cdots\cup\supp(y_n)$.
Denote  $s_i:=\supp(y_i)\setminus r_\infty$.
By Proposition~\ref{supportofsums},
\begin{equation*}
s_1\cup\cdots\cup s_n\subseteq\supp(y_1+\cdots+y_n)\subseteq r_\infty\cup s_1\cup\cdots\cup s_n=\supp(y_1)\cup\cdots\cup\supp(y_n),
\end{equation*}
and therefore it suffices to prove that $r_\infty\subseteq\supp(y_1+\cdots+y_n)$.
So let $\alpha\in r_\infty$ be arbitrary. We have already noticed
that $y_i(\alpha)=m g_\alpha$ for each $1\leq i\leq n$, where $g_\alpha$
is an element of infinite order. Consequently
\begin{equation*}
(y_1+\cdots+y_n)(\alpha)=y_1(\alpha)+\cdots+y_n(\alpha)=mg_\alpha+\cdots+mg_\alpha=(nm)g_\alpha\neq 0,
\end{equation*}
which shows that $\alpha\in\supp(y_1+\cdots+y_n)$, and we are done.
\end{proof}

We now arrive at the main technical result of this section.

\begin{theorem}\label{multicube}
Suppose that $G$ is a commutative cancellative semigroup of uncountable cardinality $\kappa$.

Then there exists a transformation $d:G\rightarrow[\kappa]^{<\omega}$
with the property that for every uncountable $X\s G$,
there exists $A\s\kappa$ such that $|A|=|X|$ and $d`` \fs(X)\supseteq[A]^{<\omega}$.
\end{theorem}
\begin{proof} We commence with an easy observation.
\begin{claim}\label{loglikefunction}
There exists a surjection $f:\omega\longrightarrow[\omega]^{<\omega}$ satisfying the two:
\begin{itemize}
\item for all $k<\omega$, $f(k)\s k$;
\item for all
$m,n<\omega$ and $\Omega\in[\omega]^{<\omega}$, there are infinitely many $k<\omega$ such that $f(m+nk)=\Omega$.
\end{itemize}
\end{claim}
\begin{proof}
For all $m,n,a,b<\omega$, the set
\begin{equation*}
D(m,n,a,b):=\{\varphi\in{}^\omega\omega\mid \exists k<\omega[k\ge a\ \&\ \varphi(m+nk)=b]\}
\end{equation*}
is dense open in the Baire space ${}^\omega\omega$.
So, by the Baire category theorem, $\bigcap\{D(m,n,a,b)\mid m,n,a,b<\omega\}\neq\emptyset$.
Pick $\varphi$ from that intersection, along with an arbitrary surjection $\psi:\omega\rightarrow[\omega]^{<\omega}$.
Then, define $f:\omega\rightarrow[\omega]^{<\omega}$, by stipulating $$f(k):=\begin{cases}(\psi\circ\varphi)(k),&\text{if }(\psi\circ\varphi)(k)\s k,\\0,&\text{otherwise}.\end{cases}$$

Clearly, $f$ is as sought.
\end{proof}

Fix $f$ as in the preceding.
For every finite set of ordinals $z$, let us denote by $\sigma_z:|z|\leftrightarrow z$ the order-preserving bijection,
so that $\sigma_z(i)$ stands for the $i^{\text{th}}$-element of $z$.

Next, embed $G$ into $\bigoplus_{\alpha<\kappa}G_\alpha$,
with each $G_\alpha$ a countable abelian group, as per
Lemma~\ref{embedding}. Then, define a colouring
$d:G\longrightarrow[\kappa]^{<\omega}$ by stipulating:
$$d(x):=\sigma_{\supp(x)}``f(|\supp(x)|).$$

To see that $d$ works, let $X$ be some uncountable subset of $G$.
The proof now splits into two cases, depending on $\lambda:=|X|$.

\underline{Case 1. Suppose that $\lambda$ is regular.}

Let $Y$ be given by Lemma~\ref{gotocondensation} with respect to $X$.
In particular, $\supp[Y]$ forms a $\Delta$-system of cardinality $\lambda$, with root, say, $r$.
Denote $m:=|r|$.
Clearly, by passing to an equipotent subset of $Y$, we may assume the existence of some positive integer $n$ and a strictly increasing function $h:m\rightarrow m+n$ such that
for every $y\in Y$:
\begin{itemize}
\item $|\supp(y)|=m+n$, and
\item $\sigma_r=\sigma_{\supp(y)}\circ h$.
\end{itemize}

Let $a$ be the maximal integer $\le m$ for which $h\restriction a$ is the identity function.
In particular, $\sigma_{\supp(y)}[a]=\sigma_r[a]$ for all $y\in Y$.

\begin{claim}\label{claim272} There exist  $b<\omega$, $\delta\le\kappa$, and a sequence $\langle y_i\mid i<\lambda\rangle$ of elements of $Y$ such that:
\begin{itemize}
\item $\{  \sigma_{\supp(y_i)}[a+b]\mid i<\lambda\}$ forms a head-tail-tail $\Delta$-system with root $\sigma_r[a]$;
\item $i\mapsto \sigma_{\supp(y_i)}(a)$ is strictly-increasing over $\lambda$;
\item $\sigma_{\supp(y_i)}[a+b]=\supp(y_i)\cap\delta$ for all $i<\lambda$.
\end{itemize}
\end{claim}
\begin{proof}
By the choice of $a$, for all $y\in Y$, we have $\min(\supp(y)\setminus r)=\sigma_{\supp(y)}(a)$.
In particular, $y\mapsto \sigma_{\supp(y)}(a)$ is injective over $Y$.
So, by the Dushnik-Miller theorem,
we may pick a sequence $\langle y_i\mid i<\lambda\rangle$ of elements of $Y$ such that $i\mapsto\sigma_{\supp(y_i)}(a)$ is strictly increasing over $\lambda$.
Put $\delta:=\sup\{\sigma_{\supp(y_i)}(a)\mid i<\lambda\}$.
Clearly, $\cf(\delta)=\lambda$.

Next, for all $j\le m+n$, let $\Lambda_j:=\{i<\lambda \mid \supp(y_i)\cap\delta= \sigma_{\supp(y_i)}[j]\}$.
This defines a partition of $\lambda$ into finitely many sets, and we may pick some positive integer $b$ such that $|\Lambda_{a+b}|=\lambda$.

Finally, recursively construct a (strictly-increasing) function $g:\lambda\rightarrow\Lambda_{a+b}$ as follows:

$\br$ Let $g(0):=\min(\Lambda_{a+b})$;

$\br$ If $i<\lambda$ is nonzero and $g\restriction i$ has already been defined,
let $\beta:=\sup_{j<i}(\supp(y_{g(j)})\cap\delta)$. By $i<\lambda=\cf(\delta)$, we have $\beta<\delta$,
so we may let $g(i):=\min\{ j\in\Lambda_{a+b} \mid \sigma_{\supp(y_j)}(a)>\beta\}$.

Clearly, $b,\delta$ and the sequence $\langle y_{g(i)}\mid i<\lambda\rangle$ are as sought.
\end{proof}

Let $b,\delta$ and $\langle y_i\mid i<\lambda\rangle$ be as in the statement of the preceding claim.
Denote $z_i:=\supp(y_i)\setminus (r\cap\delta)$. Notice that $\min(z_i)=\min(\supp(y_i)\setminus r)=\sigma_{\supp(y_i)}(a)$.

$$\frame{
\setlength{\unitlength}{0.8cm}
\begin{picture}(18.3,2)
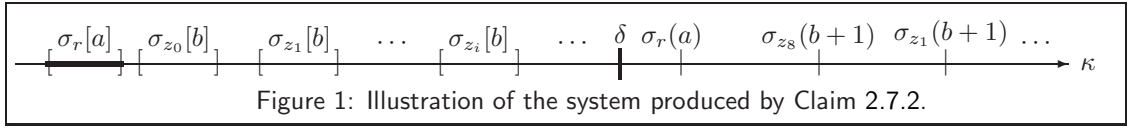
\put(4,0.25){\textsf{Figure 1: Illustration of the system produced by Claim~\ref{claim272}.}}
\put(0,1){\vector(1,0){17.5}}
\linethickness{0.6mm}
\put(17.7,0.9){$\kappa$}
\put(0.5,0.92){$[$}
\put(1.7,0.92){$]$}
\put(0.5,1){\line(1,0){1.3}}
\put(0.7,1.3){$\sigma_r[a]$}
\put(2,0.92){$[$}
\put(3.3,.92){$]$}
\put(2.2,1.3){$\sigma_{z_0}[b]$}
\put(4,0.92){$[$}
\put(5.3,0.92){$]$}
\put(4.2,1.3){$\sigma_{z_1}[b]$}
\put(6,1.25){$\cdots$}
\put(7,0.92){$[$}
\put(8.3,0.92){$]$}
\put(7.2,1.3){$\sigma_{z_i}[b]$}
\put(9,1.25){$\cdots$}
\linethickness{0.4mm}
\put(10.05,0.75){\line(0,0){0.5}}
\put(9.95,1.35){$\delta$}
\put(11,0.92){$|$}
\put(10.4,1.35){$\sigma_{r}(a)$}
\put(13.3,0.92){$|$}
\put(12.4,1.35){$\sigma_{z_8}(b+1)$}
\put(15.4,0.92){$|$}
\put(14.6,1.4){$\sigma_{z_1}(b+1)$}
\put(16.7,1.25){$\cdots$}
\end{picture}
}$$

We claim that $d``\fs(X)\supseteq[A]^{<\omega}$
for the $\lambda$-sized set $A:=\{ \min(z_i)\mid i<\lambda\}$.

To see this, fix an arbitrary $p\in[A]^{<\omega}$.
Let  $\{ \alpha_j\mid j<c\}$ be the increasing enumeration of $p$, so that $c=|p|$.
Set $\Omega:=\{ a+bj\mid j<c\}$.
By the choice of the function $f$, let us fix some integer $k> c$ such that $f(m+nk)=\Omega$.

Let $\langle i_j \mid j<k\rangle$ be a strictly increasing sequence of ordinals in $\lambda$ such that
for all $j<c$, $i_j$ is the unique ordinal to satisfy $\alpha_j=\min(z_{i_j})$.
Put  $x:=\sum_{j<k} x_j$, where $x_j:=y_{i_j}$.
By $\{ y_{i_j}\mid j<k\}\s Y$, we have:
\begin{itemize}
\item  $x\in\fs(X)$;
\item  $\supp(x)=\bigcup_{j<k} \supp(x_j)=r\uplus(\supp(x_0)\setminus r)\uplus\cdots\uplus(\supp(x_{k-1})\setminus r)$;
\item $\supp(x)\cap\delta=
\sigma_r[a]\orderunion \sigma_{\supp(x_0)}``[a,a+b]\orderunion\cdots\orderunion\sigma_{\supp(x_{k-1})}``[a,a+b].$%
\footnote{Here, $w=z\orderunion z'$ asserts that $w=z\cup z'$ and $\sup(z)<\min(z')$.}
\end{itemize}

In particular, $|\supp(x)|=m+nk>a+bc$, and
\begin{align*}
d(x)=&\sigma_{\supp(x)}``f(m+nk)=\sigma_{\supp(x)}``\Omega\\
=&\{ \sigma_{\supp(x)}(a+bj)\mid j<c\}\\
=&\{ \sigma_{\supp(x_j)}(a)\mid j<c\}=\{ \sigma_{\supp(y_{i_j})}(a)\mid j<c\}\\
=&\{ \min(z_{i_j})\mid j<c\}=\{ \alpha_j\mid j<c\}=p,
\end{align*}
as sought.

\underline{Case 2. Suppose that $\lambda$ is singular.}

Let $\langle \lambda_\gamma\mid\gamma<\cf(\lambda)\rangle$ be a strictly increasing sequence of regular cardinals
converging to $\lambda$, with $\lambda_0>\cf(\lambda)$.
Let $\langle X_\gamma\mid \gamma<\cf(\lambda)\rangle$ be a partition of $X$ with $|X_\gamma|=\lambda_\gamma$
for all $\gamma<\cf(\lambda)$.

For each $\gamma<\cf(\lambda)$,  appeal to Lemma~\ref{gotocondensation} with $X_\gamma$, to obtain a set $Y_\gamma$.
In particular, $\supp[Y_\gamma]$ forms a $\Delta$-system of cardinality $\lambda_\gamma$, with root, say, $r_\gamma$.
Let $$\delta_\gamma:=\min\{\delta\le\kappa\mid |\{\min(\supp(y)\cap(\delta\setminus r_\gamma))\mid y\in Y_\gamma\}|=\lambda_\gamma\}.$$
Clearly, $\cf(\delta_\gamma)=\lambda_\gamma$. In particular, $\gamma\mapsto\delta_\gamma$ is injective over $\cf(\lambda)$,
and we may find some cofinal subset $\Gamma\s \cf(\lambda)$ over which $\gamma\mapsto\delta_\gamma$ is strictly increasing.
Put $\delta:=\sup_{\gamma\in\Gamma}\delta_\gamma$, and $r:=\bigcup_{\gamma\in\Gamma}r_\gamma$.

For all $\gamma\in\Gamma$, by $\lambda_\gamma>\cf(\lambda)=\cf(\delta)$,
let us fix a large enough $\beta_\gamma<\delta$ for which $$Y_\gamma^0:=\{ y\in Y_\gamma\mid \min(\supp(y)\setminus r_\gamma)<\delta_\gamma\ \&\ \supp(y)\cap\delta\s\beta_\gamma\}$$ has cardinality $\lambda_\gamma$.
Let $\bar\Gamma$ be some sparse enough cofinal subset of $\Gamma$ such that $$\sup\{\beta_{\gamma'}\mid \gamma'\in\bar\Gamma\cap\gamma\}<\delta_\gamma$$ for all $\gamma'<\gamma$ both from $\bar\Gamma$.

For all $\gamma\in\Gamma$, by minimality of $\delta_\gamma$ and by $\lambda_\gamma>|[r]^{<\omega}|$, we infer that
the following set has size $\lambda_{\gamma}$:
 $$Y^1_\gamma:=\{ y\in Y^0_\gamma\mid \left(\sup\{\beta_{\gamma'}\mid \gamma'\in\bar\Gamma\cap\gamma\}<\min(\supp(y)\setminus r_\gamma)\right)\ \&\ \left(y\cap r=r_\gamma\right)\}.$$

Put $Y:=\biguplus_{\gamma\in\bar\Gamma} Y^1_\gamma$.
For all $y\in Y$, let $\gamma(y)$ denote the unique ordinal $\gamma\in\bar\Gamma$ such that $y\in Y^1_\gamma$.

Consider the following subset of $\delta$: $$A:=\{\min(\supp(y)\setminus r_{\gamma(y)})\mid y\in Y\}\setminus r.$$

\begin{claim}\label{claim273} For each $\alpha\in A$, there exists a unique $y\in Y$ such that $\alpha\in\supp(y)$.

In particular, $|A|=\lambda$.
\end{claim}
\begin{proof} Let $\alpha\in A$ be arbitrary. Fix some $y_\alpha\in Y$ such that $\alpha=\min(\supp(y_\alpha)\setminus r_{\gamma(y_\alpha)})$.
Towards a contradiction, suppose that there exists $y\in Y\setminus\{y_\alpha\}$ with $\alpha\in\supp(y)$.
There are three cases to consider, each of which leads to a contradiction:
\begin{itemize}
\item[$\br$] Suppose that $\gamma(y_\alpha)=\gamma(y)$.

Then $\alpha\in\supp(y_\alpha)\cap\supp(y)\s r_{\gamma(y)}\s r$,
contradicting the fact that $A\cap r=\emptyset$.

\item[$\br$]  Suppose that  $\gamma(y_\alpha)\in\bar\Gamma\cap \gamma(y)$.

Then  $\alpha\in\supp(y_\alpha)\cap\delta\s \beta_{\gamma(y_\alpha)}<\min(\supp(y)\setminus r_{\gamma(y)})$.

So, by $\alpha\in\supp(y)$, it must be the case that $\alpha\in r_{\gamma(y)}\s r$, contradicting the fact that $A\cap r=\emptyset$.

\item[$\br$]  Suppose that $\gamma(y)\in\bar\Gamma\cap \gamma(y_\alpha)$.

Then $\alpha\in\supp(y)\cap\delta\s\beta_{\gamma(y)}<\min(\supp(y_\alpha)\setminus r_{\gamma(y_\alpha)})=\alpha$.
This is a contradiction. \qedhere

\end{itemize}
\end{proof}

To see that $d``\fs(X)\supseteq[A]^{<\omega}$,  let  $p$ be an arbitrary element of $[A]^{<\omega}$.
For each $\alpha\in p$, pick $y_\alpha\in Y$ such that $\min(\supp(y_\alpha)\setminus r_{\gamma(y_\alpha)})=\alpha$.
Write $z:=\sum_{\alpha\in p}y_\alpha$.\footnote{If $p=\emptyset$, then  $z$ stands for $0_G$ (the identity element of $G$).}
By Claim~\ref{claim273}, we have $p\s\supp(z)$.

Fix a large enough $\gamma\in\bar\Gamma$ such that $\gamma>\gamma(y_\alpha)$ for all $\alpha\in p$.
Put $\epsilon:=\min\left(\bigcup\{\supp(y)\setminus r_\gamma\mid y\in Y_\gamma^1\}\right)$.
By the choice of $\gamma$ and the definition of $Y^1_\gamma$, we have $p\s\epsilon$.

Next, by $|\prod_{\beta\in r_\gamma} G_\beta|\le\omega$ and the pigeonhole principle, let us fix an uncountable $Y^2_\gamma\s Y^1_\gamma$, a sequence $\langle g_\beta\mid \beta\in r_\gamma\rangle$,
and a positive integer $n$, such that for all $y\in Y^2_\gamma$:
\begin{itemize}
\item $y(\beta)=g_\beta$ for all $\beta\in r_\gamma$;
\item $|\supp(y)\setminus r_\gamma|=n$.
\end{itemize}

Let $\beta\in r_\gamma$ be arbitrary.
As $Y_\gamma$ was provided by Lemma~\ref{gotocondensation} and $\beta$ belongs to the root of $\supp[Y_\gamma]$,
we get that for all $y\in Y^2_\gamma$, $g_\beta=y(\beta)$ has an infinite order.
In particular, $|\{k<\omega\mid z(\beta)+k g_\beta=0_{G_\beta}\}|\le1$.

As $r_\gamma$ is finite, let us pick a large enough $K<\omega$ such that $\{k<\omega\mid \exists \beta\in r_\gamma[z(\beta)+k g_\beta=0_{G_\beta}]\}\s K$,
so that $z(\beta)+k g_\beta\neq 0_{G_\beta}$ for all $\beta\in r_\gamma$ and $k\ge K$.

Pick an injective sequence $\langle y^i \mid i<K\rangle$ of elements of $Y^1_\gamma$.
Write $z':=z+\sum_{i<K}y^i$, and $m:=|\supp(z')|$.

Put $Y^2_\gamma:=Y^1_\gamma\setminus\{ y^i\mid i<K\}$. Recalling Claim~\ref{claim273}, we infer that for every finite $\mathcal Y\s Y^2_\gamma$:
\begin{itemize}
\item $p\uplus r_\gamma\s \supp(z'+\sum_{y\in\mathcal Y}y)$;
\item $|\supp(z'+\sum_{y\in\mathcal Y}y)|=m+n|\mathcal Y|$;
\item $\supp(z')\cap\epsilon$ is an initial segment of $\supp(z'+\sum_{y\in\mathcal Y}y)$.
\end{itemize}

By $p\s\supp(z')\cap\epsilon$, let us fix $\Omega\in[|\supp(z')\cap\epsilon|]^{<\omega}$ such that $\sigma_{\supp(z')\cap\epsilon}``\Omega=p$.
By the choice of the function $f$, let us fix some $k<\omega$ such that $f(m+nk)=\Omega$.
Pick $\mathcal Y\in[Y^2_\gamma]^k$, and set $x:=z'+\sum_{y\in\mathcal Y}y$.

Then $x\in\fs(X)$,
$|\supp(x)|=m+nk$, and
\begin{align*}
d(x)=&\sigma_{\supp(x)}``f(m+nk)=\sigma_{\supp(x)}``\Omega\\
=& (\sigma_{\supp(x)}\restriction|\supp(z')\cap\epsilon|) ``\Omega\\
=&\sigma_{\supp(z')\cap\epsilon}``\Omega=p,
\end{align*}
as sought.
\end{proof}

\begin{corollary}\label{c29} For every uncountable cardinals $\lambda\le\kappa$ and a cardinal $\theta$, the following are equivalent:
\begin{enumerate}
\item $\kappa\nrightarrow[\lambda]^{<\omega}_\theta$ holds;
\item $G\nrightarrow[\lambda]^\fs_\theta$ holds for every commutative cancellative semigroup $G$ of cardinality $\kappa$;
\item $G\nrightarrow[\lambda]^\fs_\theta$ holds for some commutative cancellative semigroup $G$ of cardinality $\kappa$.
\end{enumerate}
\end{corollary}
\begin{proof} $(1)\implies(2)$ Suppose that $\lambda,\kappa,\theta$ are as above, and that $c:[\kappa]^{<\omega}\rightarrow\theta$ witnesses $\kappa\nrightarrow[\lambda]^{<\omega}_\theta$.
That is, for every $A\in[\kappa]^\lambda$, we have $c``[A]^{<\omega}=\theta$.
Now, given a commutative cancellative semigroup $G$ of cardinality $\kappa$,
let $d:G\rightarrow[\kappa]^{<\omega}$ be given by Theorem~\ref{multicube}.
Clearly, $c\circ d$ witnesses $G\nrightarrow[\lambda]_{\theta}^\fs$.

$(2)\implies(3)$ This is trivial.

$(3)\implies(1)$ Suppose that $G$ is a commutative cancellative semigroup of cardinality $\kappa$,
and that $d:G\rightarrow\theta$ is a colouring witnessing $G\nrightarrow[\lambda]^\fs_\theta$.
Fix an injective enumeration $\{ x_\alpha\mid \alpha<\kappa\}$ of the elements of $G$,
and define $c:[\kappa]^{<\omega}\rightarrow\theta$ by stipulating $c(\{\alpha_0,\ldots,\alpha_n\}):=d(x_{\alpha_0}+\cdots+x_{\alpha_n})$.
Clearly, $c$ witnesses $\kappa\nrightarrow[\lambda]^{<\omega}_\theta$.
\end{proof}

Recall that $(\kappa,\mu)\twoheadrightarrow(\lambda,\theta)$ asserts that for every structure $(A,R,\ldots)$
for a countable first-order language with a distinguished unary predicate,
if $(|A|,|R|)=(\kappa,\mu)$,
then there exists an elementary substructure $(B,S,\ldots)\prec (A,R,\ldots)$ with $(|B|,|S|)=(\lambda,\theta)$.

To exemplify, let us point out that if $\theta$ is an infinite cardinal and there exists a $\theta^+$-Kurepa tree, then $(\theta^{++},\theta^+)\twoheadrightarrow(\theta^+,\theta)$ fails.
Also note that the instance $(\omega_2,\omega_1)\twoheadrightarrow(\omega_1,\omega)$ is known as \emph{Chang's conjecture}.

\begin{corollary}\label{chang} For every infinite regular cardinal $\theta$ and every cardinal $\kappa>\theta$, the following are equivalent:
\begin{itemize}
\item $(\kappa,\theta^+)\twoheadrightarrow(\theta^+,\theta)$ fails;
\item $G\nrightarrow[\theta^+]^\fs_{\theta^+}$ holds for every commutative cancellative semigroup $G$ of cardinality $\kappa$.
\end{itemize}
\end{corollary}
\begin{proof}
By a standard coding argument (using Skolem functions),
the failure of $(\kappa,\theta^+)\twoheadrightarrow(\theta^+,\theta)$ is equivalent  to $\kappa\nrightarrow[\theta^+]^{<\omega}_{\theta^+,\theta}$.\footnote{For a proof, see Theorem 8.1 of \cite{MR2731169}.}
Thus, recalling Corollary~\ref{c29}, it suffices to prove that
$\kappa\nrightarrow[\theta^+]^{<\omega}_{\theta^+,\theta}$ is equivalent to $\kappa\nrightarrow[\theta^+]^{<\omega}_{\theta^+}$.
Of course, only the forward implication requires an argument.
Now, as $\theta$ is regular, we get from \cite{MR908147} that  $\theta^+\nrightarrow[\theta^+]^2_{\theta^+}$ holds.
Then, as shown in \cite{MR0371662}, the conjunction of $\kappa\nrightarrow[\theta^+]^{<\omega}_{\theta^+,\theta}$ with  $\theta^+\nrightarrow[\theta^+]^2_{\theta^+}$ entails $\kappa\nrightarrow[\theta^+]^{<\omega}_{\theta^+}$.
\end{proof}

\begin{corollary}\label{c27}
For every uncountable commutative cancellative semigroup $G$,  $G\nrightarrow[\omega_1]_\omega^\fs$ holds.
\end{corollary}
\begin{proof} The map $z\mapsto|z|$ witnesses that $\kappa\nrightarrow[\omega]_\omega^{<\omega}$ holds for every infinite cardinal $\kappa$.
In particular, $\kappa\nrightarrow[\omega_1]^{<\omega}_\omega$ holds for every uncountable cardinal $\kappa$.
Now, appeal to Corollary~\ref{c29} with $\lambda=\omega_1$ and $\theta=\omega$.
\end{proof}

Modulo a large cardinal hypothesis, the preceding is optimal:
\begin{prop}\label{p29} If there exists an $\omega_1$-Erd\H{os} cardinal, then in some forcing extension,
 $\mathbb R\nrightarrow[\omega_1]^\fs_{\omega_1}$ fails.
 Furthermore, in this forcing extension, for every semigroup $(G,*)$ of size continuum and every colouring $c:G\rightarrow\omega_1$, there exists an uncountable subset $X\s G$
for which $\{c(x_0*\cdots*x_n)\mid n<\omega, x_0,\ldots,x_n\in X\}$ is countable.
\end{prop}
\begin{proof} Let $\kappa$ denote the $\omega_1$-Erd\H{o}s cardinal.
So $\kappa$ is strongly inaccessible and satisfies that for every $\theta<\kappa$ and every colouring $d:[\kappa]^{<\omega}\rightarrow\theta$,
there exists some $H\in[\kappa]^{\omega_1}$ such that $d\restriction[H]^n$ is constant for all $n<\omega$.

Let $\mathbb P$ denote the notion of forcing for adding $\kappa$ many Cohen reals. We claim that the forcing extension $V^{\mathbb P}$ is as sought.

Suppose that $\mathring{c}$ is a $\mathbb P$-name for a colouring $c:G\rightarrow\omega_1$ of a given semigroup $(G,*)$ of size continuum.
As $V^{\mathbb P}\models \mathfrak c=\kappa$, let us simplify the matter and just assume that the underlying set $G$ is in fact $\kappa$.

Working in $V$, define a colouring $d:[\kappa]^{<\omega}\rightarrow[\omega_1]^{\le\omega}$
by letting for all $n<\omega$ and all $\alpha_0<\ldots<\alpha_n<\kappa$:
$$d(\alpha_0,\ldots,\alpha_n):=\{ \delta<\omega_1\mid \exists \sigma:\{0,\ldots,n\}\rightarrow\{0,\ldots,n\}\exists p\in\mathbb P[p\Vdash``\mathring{c}({\check\alpha_{\sigma(0)}}*\cdots*\check\alpha_{\sigma(n)})=\check\delta"]\}.$$

As $\mathbb P$ is \textit{ccc}, the range of $d$ indeed consists of \emph{countable} subsets of $\omega_1$.
As $\kappa$ is an $\omega_1$-Erd\H{o}s cardinal, $|[\omega_1]^{\le\omega}|<\kappa$ and we may pick some $H\in[\kappa]^{\omega_1}$ such that $d``[H]^n$ is a singleton, say $\{A_n\}$, for every positive integer $n$.
Then $H$ is an uncountable subset of $G$, $A:=\bigcup_{n=1}^\infty A_n$ is a countable subset of $\omega_1$,
and  $$\Vdash``\forall n\in\check\omega\forall x_0,\ldots,x_n\in \check H[\mathring{c}(\check{x_0}*\cdots*\check{x_n})\in \check{A}]".\qedhere$$
\end{proof}

It is also consistent that the number of colours in Corollary~\ref{c27} \emph{may} be increased to the maximal possible value.
To see this, simply take $\theta$ to be $\omega_1$ in the next statement:
\begin{corollary}\label{cor29} It is consistent with $\zfc+\gch$ that for
every commutative cancellative semigroup $G$, $G\nrightarrow[\theta]_{\theta}^{\fs}$ holds for every uncountable cardinal $\theta$.
\end{corollary}
\begin{proof} If there exists an inaccessible cardinal in G\"odel's constructible universe $L$,
then let $\mu$ denote the least such one and work in $L_\mu$. Otherwise, work in $L$.

In both cases, we end up with a model of $\zfc+\gch$ \cite{MR0002514},
satisfying $\kappa\nrightarrow[\theta]_{\theta}^{<\omega}$ for every uncountable cardinals $\theta\le\kappa$ \cite{MR0371662}.
Now, appeal to Corollary~\ref{c29}.
\end{proof}

The preceding is quite surprising (think of the instance $|G|=\beth_{\theta+\omega}$),
but of course we are standing on the shoulders of Rowbottom \cite{MR0323572}.

\section{A set-theoretic interlude}\label{infcombinatorics}

This section is dedicated to study of the following new set-theoretic principles:

\begin{definition}\label{circledast} For infinite cardinals $\chi\le\kappa$, and an arbitrary cardinal $\theta\le\kappa$:
\begin{itemize}
\item $\axiom(\kappa,\theta)$ asserts
the existence of a colouring $d:[\kappa]^{<\omega}\rightarrow\theta$ satisfying the following.
For every $\kappa$-sized family $\mathcal X\s[\kappa]^{<\omega}$ and every $\delta<\theta$,
there exist two distinct $x,y\in\mathcal X$ such that $d(z)=\delta$ whenever $(x\symdiff y)\s z\s (x\cup y)$;

\item $\axiomi(\kappa,\theta,\chi)$ asserts
the existence of a colouring $d:\mathcal [\kappa]^{<\chi}\rightarrow\theta$ satisfying the following.
For every integer $n\ge 2$, every sequence $\langle \mathcal X_i \mid i<n\rangle\in\prod_{i<n}[[\kappa]^{<\chi}]^\kappa$,
and every $\delta<\theta$,
there exists $\langle x_i\mid i<n\rangle\in\prod_{i<n}\mathcal X_i$ such that $\langle \sup(x_i)\mid i<n\rangle$ is strictly increasing,
and such that $d(z)=\delta$ whenever
\begin{equation*}
\bigcup_{i<n}\left(x_i\setminus\bigcup_{j\in n\setminus\{i\}}x_j\right)\s z\s \bigcup_{i<n}x_i.
\end{equation*}

\end{itemize}
\end{definition}

Hindman's theorem is well-known to be equivalent to a Ramsey-theoretic statement concerning block sequences.
Thus, the reader may want to observe that whenever $d:[\kappa]^{<\omega}\rightarrow\theta$ witnesses  $\axiom(\kappa,\theta)$,
then for every block sequence $\vec{x}=\langle x_\alpha \mid \alpha<\kappa\rangle$ of finite subsets of $\kappa$ (i.e., satisfying $\max(x_\alpha)<\min(x_\beta)$ for all $\alpha<\beta<\kappa$), we have $d``\fu(\vec{x})=\theta$ (indeed,
for every $\delta<\theta$, there exist  $\alpha<\beta<\kappa$
such that $d(x_{\alpha}\cup x_{\beta})=\delta$).
Observe that if $d$ furthermore witnesses $\axiomi(\kappa,\theta,\omega)$,
then for every positive $n<\omega$ and every $\delta<\theta$,
there exist $\alpha_0<\dots<\alpha_n<\kappa$ such that $d(x_{\alpha_0}\cup\cdots\cup x_{\alpha_n})=\delta$.

\begin{prop}\label{prop32} For all infinite $\chi\le\chi'\le\kappa$, and all $\theta\le\theta'\le\kappa$:
\begin{enumerate}
\item $\axiomi(\kappa,\theta',\chi')$ entails $\axiomi(\kappa,\theta,\chi)$;
\item $\axiomi(\kappa,\theta,\chi)$ entails $\kappa\nrightarrow[\kappa;\kappa]^2_\theta$;
\item $\axiomi(\kappa,\theta,\chi)$ entails $\axiom(\kappa,\theta)$;
\item $\axiom(\kappa,\theta)$ entails $\kappa\nrightarrow[\kappa]^2_\theta$.
\end{enumerate}
\end{prop}
\begin{proof} (1) This is obvious.

(2) Let $d:[\kappa]^{<\chi}\rightarrow\theta$ be a witness to $\axiomi(\kappa,\theta,\chi)$.
Define $c:[\kappa]^2\rightarrow\theta$ by letting $c(\alpha,\beta):=d(\{\alpha,\beta\})$.
Now, suppose that we are given $X,Y\in[\kappa]^\kappa$.
We need to verify that for all $\delta<\theta$, there exist $\alpha\in X$ and $\beta\in Y$ such that $\alpha<\beta$ and $c(\alpha,\beta)=\delta$.
Put $\mathcal X:=\{\{\alpha\}\mid\alpha\in X\}$ and $\mathcal Y:=\{\{\beta\}\mid \beta\in Y\}$.
Fix $x\in\mathcal X$ and $y\in\mathcal Y$ such that $\sup(x)<\sup(y)$ and $d(z)=\delta$ whenever $(x\symdiff y)\s z\s (x\cup y)$.
Let $\alpha:=\sup(x)$ and $\beta:=\sup(y)$. Then $\alpha\in X$, $\beta\in Y$, $\alpha<\beta$ and $c(\alpha,\beta)=\delta$, as sought.

(3) Let $d:\mathcal [\kappa]^{<\chi}\rightarrow\theta$ be a witness to $\axiomi(\kappa,\theta,\chi)$.
We claim that $d\restriction[\kappa]^{<\omega}$ witnesses $\axiom(\kappa,\theta)$.
To avoid trivialities, suppose that $\theta>1$.
In particular, by Clause~(2) and Ramsey's theorem, $\kappa$ is uncountable.

Given a $\kappa$-sized family $\mathcal X\s[\kappa]^{<\omega}$,
pick a partition $\mathcal X=\mathcal X_0\uplus\mathcal X_1$ with $|\mathcal X_0|=|\mathcal X_1|=\kappa$.
Then, by the choice of $d$, for every $\delta<\theta$, there exist $x\in \mathcal X_0$ and $y\in \mathcal X_1$ such that $d(z)=\delta$
whenever $(x\symdiff y)\s z\s (x\cup y)$. Clearly, $x,y$ are distinct elements of $\mathcal X$.

(4) Similar to the proof of Clause~(2).
\end{proof}

\begin{prop}\label{singulars} Suppose that $\lambda>\cf(\lambda)=\kappa$ are infinite cardinals.
Then for every cardinals $\theta,\chi$:
\begin{enumerate}
\item  $\axiom(\kappa,\theta)$ entails  $\axiom(\lambda,\theta)$;
\item  $\axiomi(\kappa,\theta,\chi)$ entails  $\axiomi(\lambda,\theta,\chi)$, provided that $\mu^{<\chi}<\lambda$ for every cardinal $\mu<\lambda$.
\end{enumerate}
\end{prop}
\begin{proof} Pick a club $\Lambda$ in $\lambda$ with $\otp(\Lambda)=\kappa$,
and derive a mapping $h:\lambda\rightarrow\kappa$ by stipulating $h(\alpha):=\otp(\Lambda\cap\alpha)$.

(1) Let $d:[\kappa]^{<\omega}\rightarrow\theta$ be a witness to $\axiom(\kappa,\theta)$.
Define $d_h:[\lambda]^{<\omega}\rightarrow\theta$ by stipulating $d_h(z):=d(h[z])$.

To see that $d_h$ witnesses $\axiom(\lambda,\theta)$, suppose that we are given a $\lambda$-sized family $\mathcal X\s[\lambda]^{<\omega}$,
and a prescribed colour $\delta<\theta$. Put $\mathcal X_h:=\{h[x]\mid x \in\mathcal X\}$.
As $|[\mu]^{<\omega}|<\lambda=|\mathcal X|$ for all $\mu\in\Lambda$, we infer that $\mathcal X_h$ is a $\kappa$-sized subfamily of $[\kappa]^{<\omega}$.
Thus, by the choice of $d$, we may pick two distinct $x',y'\in\mathcal X_h$ such that $d(z')=\delta$ whenever $(x'\symdiff y')\s z'\s (x'\cup y')$.
Now, find $x,y\in\mathcal X$ such that $h[x]=x'$ and $h[y]=y'$. Clearly, $x$ and $y$ are distinct.
Finally, suppose that $z$ is some set satisfying $(x\symdiff y)\s z\s (x\cup y)$.
Then $(x'\symdiff y')\s h[z]\s (x'\cup y')$, and hence $d_h(z)=d(h[z])=\delta$, as sought.

(2) Let $d:[\kappa]^{<\chi}\rightarrow\theta$ be a witness to $\axiomi(\kappa,\theta,\chi)$.
Define $d_h:[\lambda]^{<\chi}\rightarrow\theta$ by stipulating $d_h(z):=d(h[z])$.
Note that for every $\lambda$-sized family $\mathcal X\s[\lambda]^{<\chi}$,
 $\mathcal X_h:=\{h[x]\mid x \in\mathcal X\}$ is a $\kappa$-sized sufamily of $[\lambda]^{<\chi}$,
 because $|[\mu]^{<\chi}|<\lambda=|\mathcal X|$ for all $\mu\in\Lambda$.
The rest of the verification is similar to that of Clause~(1).
\end{proof}

Recall that $\pr_1(\kappa,\kappa,\theta,\chi)$ asserts the existence of a
colouring $c:[\kappa]^2\rightarrow\theta$ satisfying that
for every $\gamma<\theta$ and every $\mathcal A\s[\kappa]^{<\chi}$ of size $\kappa$,
consisting of pairwise disjoint sets, there exist  $x,y\in\mathcal A$ with $\sup(x)<\min(y)$ for which $c[x\times y]=\{\gamma\}$.

\begin{lemma}\label{pr1lemma} Suppose that $\pr_1(\kappa,\kappa,\theta,\chi)$ holds for given infinite cardinals $\chi\le\theta\le\kappa=\cf(\kappa)$.

If $\kappa$ is uncountable, and $\mu^{<\chi}<\kappa$ for every cardinal $\mu<\kappa$, then  $\axiomi(\kappa,\theta,\chi)$ holds.
\end{lemma}
\begin{proof}
Let $c:[\kappa]^2\rightarrow\theta$ be a witness to $\pr_1(\kappa,\kappa,\theta,\chi)$.
Fix a bijection $\pi:\theta\leftrightarrow\theta\times\chi$. Define $c_0:[\kappa]^2\rightarrow\theta$ and $c_1:[\kappa]^2\rightarrow\chi$
in such a way that if $c(\alpha,\beta)=\gamma$ and $\pi(\gamma)=\pair{\delta,\epsilon}$, then $c_0(\alpha,\beta)=\delta$ and $c_1(\alpha,\beta)=\epsilon$.

Now, define $d:\mathcal [\kappa]^{<\chi}\rightarrow\theta$ as follows. Let $z\in[\kappa]^{<\chi}$ be arbitrary.
If $M_z:=\{ \pair{\alpha,\beta}\in[z]^2\mid c_1(\alpha,\beta)=\sup(c_1``[z]^2)\}$ is nonempty,
then let $d(z):=c_0(\alpha,\beta)$ for an arbitrary choice of $\pair{\alpha,\beta}$ from $M_z$.
Otherwise, let $d(z):=0$.

To see that $d$ works, suppose we are given a sequence $\langle \mathcal X_i \mid i<n\rangle\in\prod_{i<n}[[\kappa]^{<\chi}]^\kappa$,
for some integer $n\ge2$, along with some prescribed colour $\delta<\theta$ .

By thinning-out, we may assume that for all $i<n$, $\mathcal X_i$ forms a $\Delta$-system with root, say, $r_i$.\footnote{This is where we use the hypothesis that $\mu^{<\chi}<\kappa$ for every cardinal $\mu<\kappa$.}
In particular, for all $i<n$, $\{ x\setminus r_i\mid x\in\mathcal X_i\}$ consists of $\kappa$ many pairwise disjoint bounded subsets of $\kappa$.
Consequently, we can construct (e.g., by recursion on $\gamma<\kappa$) a matrix $\langle x^\gamma_i \mid \gamma<\kappa, i<n \rangle$ in such a way that for all $\gamma<\gamma'<\kappa$ and $i<j<n$:
\begin{itemize}
\item $r_i\uplus x^\gamma_i \in\mathcal X_i$;
\item $\sup(r_0\cup\cdots\cup r_{n-1})<\min(x^\gamma_i)\le\sup(x^\gamma_i)<\min(x^\gamma_j)\le\sup(x_j^\gamma)<\min(x_0^{\gamma'})$.
\end{itemize}

By the pigeonhole principle, let us fix $\Gamma\in[\kappa]^\kappa$ and $\epsilon<\chi$ such that for all $\gamma\in \Gamma$:
\begin{itemize}
\item $\sup(c_1``[r_0\cup\cdots\cup r_{n-1}\cup x^\gamma_0\cup\cdots\cup x^\gamma_{n-1}]^2)=\epsilon$.
\end{itemize}

Denote $a_\gamma:=x_0^\gamma\uplus\cdots\uplus x_{n-1}^\gamma$. As $\mathcal A:=\{ a_\gamma\mid \gamma\in \Gamma\}$
is a $\kappa$-sized subfamily of $[\kappa]^{<\chi}$ consisting of pairwise disjoint sets,
we may now pick $\gamma<\gamma'$ both from $\Gamma$ for which $c[a_\gamma\times a_{\gamma'}]=\{\pi^{-1}\pair{\delta,\epsilon+1}\}$.

Write $\bar x_0:=r_0\uplus x_0^\gamma$,
and for all nonzero $i<n$, write $\bar x_i:=r_i\uplus x_i^{\gamma'}$.
Clearly, $\langle \bar x_i\mid i<n\rangle\in\prod_{i<n} \mathcal X_i$, and $\langle \sup(\bar x_i)\mid i<n\rangle$ is strictly increasing.
Next, suppose that we are given $z$ satisfying
$$\bigcup_{i<n}\left(\bar x_i\setminus\bigcup_{j\in n\setminus\{i\}}\bar x_j\right)\s z\s \bigcup_{i<n}\bar x_i.$$

\begin{claim} $M_z$ is a nonempty subset of $a_\gamma\times a_{\gamma'}$.
\end{claim}
\begin{proof} Let $\pair{\alpha,\beta}\in [z]^2$ be arbitrary.
As $\bigcup_{i<n}\bar x_i=\left(\bigcup_{i<n} r_i\right)\uplus x_0^\gamma\uplus \left(\biguplus_{0<i<n} x_i^{\gamma'}\right)$,
we consider the following cases:
\begin{enumerate}
\item Suppose that $\alpha\in \left(\bigcup_{i<n} r_i\right)$.

By $\beta\in(r_0\cup\cdots\cup r_{n-1}\cup x_0^\gamma\cup x_1^{\gamma'}\cup x_{n-1}^{\gamma'})$, we have $c_1(\alpha,\beta)\le\epsilon$.
\item Suppose that $\alpha\in x_0^\gamma$.
\begin{enumerate}
\item If $\beta\in(r_0\cup\cdots\cup r_{n-1}\cup x_0^\gamma)$, then $c_1(\alpha,\beta)\le\epsilon$;
\item If $\beta\in \biguplus_{0<i<n} x_i^{\gamma'}$, then $\pair{\alpha,\beta}\in a_\gamma\times a_{\gamma'}$ and hence $c(\alpha,\beta)=\pi^{-1}\pair{\delta,\epsilon+1}$,
so that $c_1(\alpha,\beta)=\epsilon+1$.
\end{enumerate}
\item Suppose that $\alpha\in \biguplus_{0<i<n} x_i^{\gamma'}$.
\begin{enumerate}
\item If $\beta\in x_0^\gamma$, then $\alpha>\beta$ which gives a contradiction to $\pair{\alpha,\beta}\in[z]^2$;
\item If $\beta\in r_0\cup\cdots\cup r_{n-1}\cup x_1^{\gamma'}\cup\cdots x_{n-1}^{\gamma'}$, then $c_1(\alpha,\beta)\le\epsilon$.
\end{enumerate}
\end{enumerate}

Finally, by $$z\supseteq
\bigcup_{i<n}\left(\bar x_i\setminus\bigcup_{j\in n\setminus\{i\}}\bar x_j\right)\supseteq
\left(\bigcup_{i<n}\bar x_i\setminus\bigcup_{i<n}r_i\right),$$
we have $x_0^\gamma\times x_1^{\gamma'}\s [z]^2$, so that case (2)(b)
is indeed feasible. Consequently, $M_z$ is a nonempty subset of $a_i\times a_j$.
\end{proof}

Let $\pair{\alpha,\beta}\in M_z$ be such that $d(z)=c_0(\alpha,\beta)$.
By the preceding claim, $c(\alpha,\beta)=\pi^{-1}\pair{\delta,\epsilon+1}$, so that $d(z)=c_0(\alpha,\beta)=\delta$, as sought.
\end{proof}

The colouring principle $\pr_1(\cdots)$ was studied extensively by many authors,
including Eisworth, Galvin, Rinot, Shelah, and Todor\v{c}evi\'c.
To mention a few results:

\begin{fact}\label{factpr1}
Suppose that $\kappa$ is a regular uncountable cardinal.

Then $\pr_1(\kappa,\kappa,\theta,\chi)$ holds in all of the following cases:
\begin{enumerate}
\item $\kappa=\theta=\mathfrak b=\omega_1$ and $\chi=\omega$;
\item $\kappa=\theta>\chi^+$, and $\square(\kappa)$ holds;\footnote{The definition of $\square(\kappa)$ may be found in \cite[p.~267]{MR908147}.}
\item\label{nonreflecting} $\kappa=\theta>\chi^+$, and $E^\kappa_{\ge\chi}$ admits a non-reflecting stationary set;
\item\label{successor_of_regular} $\kappa=\theta=\lambda^+>\chi^+$, and $\lambda$ is regular;
\item\label{pp} $\kappa=\theta=\lambda^+$, $\lambda$ is singular, $\chi=\cf(\lambda)$, and $\pp(\lambda)=\lambda^+$ (e.g., $\lambda^{\cf(\lambda)}=\lambda^+$);\footnote{For a light introduction to $\pp(\lambda)$, see \cite{rinot05}.}
\item $\kappa=\theta=\lambda^+$, $\lambda$ is singular, $\chi=\cf(\lambda)$, and there exists a collection of $<\cf(\lambda)$ many stationary subsets of $\kappa$ that do not reflect simultaneously;
\item\label{successor_of_singular} $\kappa=\lambda^+$, $\lambda$ is singular, and $\theta=\chi=\cf(\lambda)$.
\end{enumerate}
\end{fact}
\begin{proof} \begin{enumerate}
\item By Lemma 1.0 of \cite[$\S1$]{MR980949}.
\item By Theorem~B of \cite{rinot18}.
\item By Corollary~3.2 of \cite{rinot15}.
\item By Clause~(\ref{nonreflecting}) above.
\item By Corollary~6.2 of \cite{MR3087059}.
\item By Corollary~3.3 of \cite{rinot13}.
\item By Conclusion~4.1 of \cite{Sh:355}. \qedhere
\end{enumerate}
\end{proof}

It follows that $\axiomi(\lambda^+,\lambda^+,\omega)$ holds for every regular cardinal $\lambda\ge\omega_1$.
Now, what about $\axiomi(\omega_1,\omega_1,\omega)$?

$\br$  Galvin proved \cite{MR0585558} that  the failure of $\pr_1(\omega_1,\omega_1,\omega_1,\omega)$ is consistent with $\zfc$,
and hence Lemma~\ref{pr1lemma} is inapplicable here.

$\br$ Getting just  $\axiom(\omega_1,\omega_1)$ turns out to be ready-made;
It follows from Theorem 2.6 of \cite{rinot15} that  $\axiom(\mu^+,\mu^+)$ holds for every infinite cardinal $\mu$ satisfying $\mu^{<\mu}=\mu$.

Altogether, there is a need for a dedicated proof of $\axiomi(\omega_1,\omega_1,\omega)$. This is our next task.

\begin{theorem}\label{lowertrace}  $\axiomi(\omega_1,\omega_1,\omega)$ holds.
\end{theorem}
\begin{proof}
As the product of $\mathfrak c$ many separable topological spaces is again separable,
let us pick a countable dense subset $\{f_\iota\mid \iota<\omega\}$ of the product space $\omega^{\omega_1}$.
Notice that this means that for every finite subset $a\s\omega_1$ and every function $f:a\rightarrow\omega$, there exists some $\iota<\omega$ such that $f_\iota\restriction a=f$.

Next, by Theorem~1.5 of \cite{MR2220104}, let us pick a function $\osc:[\omega_1]^2\rightarrow\omega$ satisfying that
for every positive integers $k,l$, every uncountable families $\mathcal A\s[\omega_1]^k$ and $\mathcal B\s[\omega_1]^l$,
each consisting of pairwise disjoint sets, and every $s<\omega$, there exist $a\in\mathcal A$ and a sequence $\langle b_m\mid m<s\rangle$ of elements in $\mathcal B$
such that for all $m<s$:
\begin{itemize}
\item $\max(a)<\min(b_m)$, and
\item $\osc(a(i),b_m(j))=\osc(a(i),b_0(j))+m$ for all $i<k$ and $j<l$.\footnote{Here, $a(i)$ stands for the unique $\alpha\in a$ to satisfy $|a\cap\alpha|=i$. The interpretation of $b_m(j)$ is similar.}
\end{itemize}

For every nonzero $\alpha<\omega_1$, fix a surjection $\psi_\alpha:\omega\rightarrow\alpha$.
Finally, define the function  $d:[\omega_1]^{<\omega}\rightarrow\omega_1$ as follows.
Let $z\in[\omega_1]^{<\omega}$ be arbitrary.
If $M_z:=\{ \pair{\alpha,\beta}\in[z]^2\mid \osc(\alpha,\beta)=\sup(\osc``[z]^2)\}$ is empty, then let $d(z):=\emptyset$.
Otherwise, pick an arbitrary $\pair{\alpha,\beta}$ from $M_z$,
let $\iota$ be the maximal natural number to satisfy that $2^\iota$ divides $\osc(\alpha,\beta)$,
and then put $d(z):=\psi_\alpha(f_\iota(\alpha))$.

To see that $d$ works, suppose we are given a sequence $\langle \mathcal X_i \mid i<n\rangle\in\prod_{i<n}[[\omega_1]^{<\omega}]^{\omega_1}$,
for some integer  $n\ge2$, along with some prescribed colour $\delta<\omega_1$.
As in the proof of Lemma~\ref{pr1lemma},
we may find a matrix $\langle x^\gamma_i \mid \gamma<\omega_1, i<n \rangle$ and a sequence $\langle r_i\mid i<n\rangle$ such that for all $\gamma<\gamma'<\omega_1$ and $i<j<n$:
\begin{itemize}
\item $r_i\uplus x^\gamma_i \in\mathcal X_i$;
\item $\sup(r_0\cup\cdots\cup r_{n-1})<\min(x^\gamma_i)\le\max(x^\gamma_i)<\min(x^\gamma_j)\le\max(x_j^\gamma)<\min(x_0^{\gamma'})$.
\end{itemize}

For all $\gamma<\omega_1$, denote $a^\gamma:=x_0^\gamma$ and $b^\gamma:=x_1^\gamma\uplus\cdots\uplus x_{n-1}^\gamma$.
By the pigeonhole principle, let us fix an uncountable $\Gamma\s\omega_1$ along with $k,l,\iota,\epsilon<\omega$ such that for all $\gamma\in\Gamma$:
\begin{itemize}
\item $|a^\gamma|=k$ and $|b^\gamma|=l$;
\item $\psi_\alpha(f_\iota(\alpha))=\delta$ for all $\alpha\in a^\gamma$;
\item $\max(\osc``[r_0\cup\cdots\cup r_{n-1}\cup a^\gamma\cup b^\gamma]^2)=\epsilon$.
\end{itemize}

Put $s:=\epsilon+1+2^\iota+1$. Consider $\mathcal A:=\{ a^\gamma\mid \gamma\in\Gamma\}$ and $\mathcal B:=\{b^\gamma\mid \gamma\in\Gamma\}$.
As $\mathcal A\s[\omega_1]^k$ and $\mathcal B\s[\omega_1]^l$ are uncountable families, each consisting of pairwise disjoint sets,
we may pick $a\in\mathcal A$ and a sequence $\langle b_m\mid m<s\rangle$ of elements in $\mathcal B$
such that for all $m<s$:
\begin{itemize}
\item $\max(a)<\min(b_m)$, and
\item $\osc(a(i),b_m(j))=\osc(a(i),b_0(j))+m$ for all $i<k$ and $j<l$.
\end{itemize}

Write $\mathfrak m:=\max(\osc([a\times b_0]))$.
Let $t$ be the unique natural number to satisfy
$0\le t<2^\iota$ and $\mathfrak m+\epsilon+1\equiv t\pmod{2^\iota}$.
Put $m:=\epsilon+1+2^\iota-t$. Then $m<s$ and $\iota$ is the maximal natural number to satisfy that $2^\iota$ divides $\mathfrak m+m$.

Fix $\gamma<\gamma'<\omega_1$ such that $a=a^\gamma$ and $b_m=b^{\gamma'}$.
Write $\bar x_0:=r_0\uplus x_0^\gamma$,
and for all nonzero $i<n$, write $\bar x_i:=r_i\uplus x_i^{\gamma'}$.
Clearly, $\langle \bar x_i\mid i<n\rangle\in\prod_{i<n} \mathcal X_i$, and $\langle \sup(\bar x_i)\mid i<n\rangle$ is strictly increasing.
Next, suppose that we are given $z$ satisfying
$$\bigcup_{i<n}\left(\bar x_i\setminus\bigcup_{j\in n\setminus\{i\}}\bar x_j\right)\s z\s \bigcup_{i<n}\bar x_i.$$

\begin{claim} $M_z$ is a nonempty subset of $a\times b_m$,
and $\osc(\alpha,\beta)=\mathfrak m+m$ for all $\pair{\alpha,\beta}\in M_z$.
\end{claim}
\begin{proof} Let $\pair{\alpha,\beta}\in [z]^2$ be arbitrary.
As $\bigcup_{i<n}\bar x_i=\left(\bigcup_{i<n} r_i\right)\uplus a^\gamma\uplus b^{\gamma'}$,
we consider the following cases:
\begin{enumerate}
\item Suppose that $\alpha\in \left(\bigcup_{i<n} r_i\right)$.

By $\beta\in(r_0\cup\cdots\cup r_{n-1}\cup a^\gamma\cup b^{\gamma'})$, we have $\osc(\alpha,\beta)\le\epsilon$.
\item Suppose that $\alpha\in a^\gamma$.
\begin{enumerate}
\item If $\beta\in(r_0\cup\cdots\cup r_{n-1}\cup a^\gamma)$, then $\osc(\alpha,\beta)\le\epsilon$;
\item If $\beta\in b^{\gamma'}$, then $\pair{\alpha,\beta}\in a\times b_m$,
so that writing $i:=a\cap\alpha$ and $j:=b_m\cap\beta$,
we have $\osc(\alpha,\beta)=\osc(a(i),b_m(j))=\osc(a(i),b_0(j))+m\ge m>\epsilon$.

\end{enumerate}
\item Suppose that $\alpha\in b^{\gamma'}$.
\begin{enumerate}
\item If $\beta\in a^\gamma$, then $\alpha>\beta$ which gives a contradiction to $\pair{\alpha,\beta}\in[z]^2$;
\item If $\beta\in r_0\cup\cdots\cup r_{n-1}\cup b^{\gamma'}$, then $\osc(\alpha,\beta)\le\epsilon$.
\end{enumerate}
\end{enumerate}

Finally, by $$z\supseteq
\bigcup_{i<n}\left(\bar x_i\setminus\bigcup_{j\in n\setminus\{i\}}\bar x_j\right)\supseteq
\left(\bigcup_{i<n}\bar x_i\setminus\bigcup_{i<n}r_i\right),$$
we have $a^\gamma\times b^{\gamma'}\s [z]^2$, so that case (2)(b)
is indeed feasible, and $M_z$ is a nonempty subset of $a^\gamma\times b^{\gamma'}=a\times b_m$.
It follows that  $M_z=\{ \pair{\alpha,\beta}\in a\times b_m\mid \osc(\alpha,\beta)=\mathfrak m+m\}$.
\end{proof}

Let $\pair{\alpha,\beta}\in M_z$ be arbitrary. By the choice of $m$, we know that $\iota$ is the maximal natural number to satisfy that $2^\iota$ divides $\mathfrak m+m$,
and hence $d(z)=\psi_\alpha(f_\iota(\alpha))=\delta$, as sought.
\end{proof}

\begin{corollary}\label{c36} Suppose that $\kappa$ is a regular uncountable cardinal that admits a nonreflecting stationary set (e.g., $\kappa$ is the successor of an infinite regular cardinal). Then:
\begin{itemize}
\item $\axiomi(\kappa,\kappa,\omega)$ holds. In particular:
\item There exists a colouring $d:[\kappa]^{<\omega}\rightarrow\kappa$
such that for every $\mathcal X\subseteq[\kappa]^{<\omega}$ of size $\kappa$ and every
colour $\delta<\kappa$, there exist two distinct $x,y\in \mathcal X$
satisfying $d(x\cup y)=\delta$.
\end{itemize}
\end{corollary}
\begin{proof}
For $\kappa=\aleph_1$, use Theorem~\ref{lowertrace}.
For $\kappa>\aleph_1$, use Lemma~\ref{pr1lemma} together with Fact~\ref{factpr1}(\ref{nonreflecting}).
\end{proof}

Note that the second bullet of the preceding generalizes the celebrated result from \cite[p.~285]{MR908147} asserting that $\kappa\nrightarrow[\kappa]^2_\kappa$ holds
for every regular uncountable cardinal $\kappa$ that admits a nonreflecting stationary set.

\medskip

As colouring of the real line is of a special interest, and as the results so far only shed a limited amount of light on cardinals of the form $2^\lambda$,
our next task is proving the following.

\begin{theorem}\label{power} Suppose that $\lambda$ is an infinite cardinal satisfying  $2^{<\lambda}=\lambda$. Then $\axiom(\cf(2^\lambda),\omega)$ holds.
\end{theorem}
\begin{proof}
By $2^{<\lambda}=\lambda$ and a classic theorem of Sierpi\'nski, there exists a linear ordering of size $2^\lambda$ with a dense subset of size $\lambda$.
Then, by Theorem~3 of \cite{MR799818} (independently, also by the main result of \cite{MR776283}),
we may fix a linear order $(L,\lhd)$ of size $\kappa:=\cf(2^\lambda)$ which is $\kappa$-entangled.
The latter means that for every  $\mu<\omega$, every $\Omega\s\mu$,
and every injective sequence $\langle f_\alpha:\mu\rightarrow L\mid \alpha<\kappa\rangle$ of order-preserving maps from $(\mu,\in)$ to $(L,\lhd)$, with pairwise disjoint images,
there exist $\alpha<\beta<\kappa$  such that for all $\iota<\mu$: $f_\alpha(\iota) \lhd f_\beta(\iota)$ iff $\iota\in\Omega$.%
\footnote{Note that the definition of a $(\kappa,\mu)$-entangled ordering in \cite{MR799818} only guarantees ``$\alpha\neq\beta$'',
however, ``$\alpha<\beta$'' can be ensured by appealing to the $(\kappa,2\mu)$-entangledness of the ordering.}

Fix a sequence of injections $\langle l_\alpha:\omega\rightarrow L\mid \alpha<\kappa\rangle$
such that  $\im(l_\alpha)\cap\im(l_\beta)=\emptyset$ for all $\alpha<\beta<\kappa$.
Define a colouring $c:[\kappa]^2\rightarrow\omega$ as follows. For all $\alpha<\beta<\kappa$,
if there exists some $\tau<\omega$ such that $l_\alpha(\tau)\lhd l_\beta(\tau)$, let $c(\alpha,\beta)$ be the least such $\tau$.
Otherwise, let $c(\alpha,\beta):=0$.

By $\kappa\le2^\lambda$, let $\langle g_\alpha:\lambda\rightarrow2\mid \alpha<\kappa\rangle$ be a sequence of pairwise distinct functions.
For all $\alpha<\beta<\kappa$, let $\Delta(\alpha,\beta):=\min\{\varepsilon<\lambda\mid g_\alpha(\varepsilon)\neq g_\beta(\varepsilon)\}$ .
Now, to define $d:[\kappa]^{<\omega}\rightarrow\omega$, let $z\in[\kappa]^{<\omega}$ be arbitrary.
If $M_z:=\{\pair{\alpha,\beta}\in[z]^2\mid \Delta(\alpha,\beta)=\sup(\Delta``[z]^2) \}$
is nonempty, then let $d(z):=c(\alpha,\beta)$ for an arbitrary choice of $\pair{\alpha,\beta}$ from $M_z$.
Otherwise, let $d(z):=0$.

To see that $d$ witnesses $\axiom(\kappa,\omega)$, suppose that we are given a $\kappa$-sized family $\mathcal X\s[\kappa]^{<\omega}$ and a prescribed colour $\delta<\omega$.
By the $\Delta$-system lemma, we may find a sequence $\langle a_\gamma \mid \gamma<\kappa\rangle$ along with $r\in[\kappa]^{<\omega}$ and $m<\omega$ such that for all $\gamma<\gamma'<\kappa$:
\begin{itemize}
\item $|a_\gamma|=m$;
\item $r\uplus a_\gamma \in\mathcal X$;
\item $\sup(r)<\min(a_\gamma)\le\max(a_\gamma)<\min(a_{\gamma'})$.
\end{itemize}

For each $\gamma<\kappa$, let $f_\gamma:m(\delta+1)\rightarrow L$ denote the unique order-preserving map from $(m(\delta+1),\in)$ to $(L,\lhd)$
such that $\im(f_\gamma)=\{l_\alpha(\tau)\mid \alpha\in a_\gamma, \tau\le\delta\}$.
Also, fix some enumeration $\{ a_\gamma(j)\mid j<m\}$ of $a_\gamma$.

Next, by an iterated application of the pigeonhole principle, let us fix $\Gamma\in[\kappa]^\kappa$ together with
 $\epsilon<\lambda$, $t:m\rightarrow{}^{(\epsilon+1)}2$ and $h:m\times(\delta+1)\leftrightarrow m (\delta+1)$ such that for all $\gamma\in \Gamma$:
\begin{itemize}
\item $\max(\Delta``[r\uplus a_\gamma]^2)=\epsilon$;
\item $\langle g_{a_\gamma(j)}\restriction (\epsilon+1)\mid j<m\rangle=t$;\footnote{Note that $|{}^m({}^{(\epsilon+1)}2)|\le 2^{<\lambda}=\lambda<\kappa$.}
\item $f_\gamma(h(j,\tau))=l_{a_\gamma(j)}(\tau)$ for all $j<m$ and $\tau\le\delta$.
\end{itemize}

Put $\Omega:=h[m\times\{\delta\}]$.
As $(L,\lhd)$ is $\kappa$-entangled, let us pick $\gamma<\gamma'$ both from $\Gamma$ such that for all $\iota<m(\delta+1)$: $f_\gamma(\iota) \lhd f_{\gamma'}(\iota)$ iff $\iota\in\Omega$.
Write $x:=r\uplus a_\gamma$ and $y:=r\uplus a_{\gamma'}$.
Clearly, $x,y$ are two distinct elements of $\mathcal X$.
Next, suppose that we are given $z$ satisfying $x\symdiff y\s z \s x\cup y$.

\begin{claim} $M_z$ is a nonempty subset of $\{\pair{a_\gamma(j),a_{\gamma'}(j)}\mid j<m\}$.
\end{claim}
\begin{proof} Let $\pair{\alpha,\beta}\in[z]^2$ be arbitrary.
As $x\cup y=r\uplus a_\gamma\uplus a_{\gamma'}$, we consider the following cases:
\begin{enumerate}
\item Suppose that $\alpha\in r$.

By $\beta\in r\uplus a_\gamma\uplus a_{\gamma'}$, we have $\Delta(\alpha,\beta)\le\epsilon$.
\item Suppose that $\alpha\in a_\gamma$.
\begin{enumerate}
\item If $\beta\in(r\uplus a_\gamma)$, then $\Delta(\alpha,\beta)\le\epsilon$;
\item If $\beta\in a_{\gamma'}$, then let $j_\alpha,j_\beta<m$ be such that $\alpha=a_\gamma(j_\alpha)$ and $\beta=a_{\gamma'}(j_\beta)$.
There are two cases to consider:
\begin{enumerate}
\item If $j_\alpha=j_\beta$, then $g_\alpha\restriction(\epsilon+1)=g_{a_\gamma(j_\alpha)}\restriction(\epsilon+1)=t(j_\alpha)=t(j_\beta)=g_{a_{\gamma'}(j_\beta)}\restriction(\epsilon+1)=g_\beta\restriction(\epsilon+1)$,
so that $\Delta(\alpha,\beta)>\epsilon$;
\item If $j_\alpha\neq j_\beta$, then $g_\alpha\restriction(\epsilon+1)=t(j_\alpha)=g_{a_{\gamma'}(j_\alpha)}\restriction(\epsilon+1)$, and hence
$$\Delta(\alpha,a_{\gamma'}(j_\alpha))>\epsilon=\max(\Delta``[a_{\gamma'}]^2)\ge\Delta(a_{\gamma'}(j_\alpha),\beta),$$
so that $\Delta(\alpha,\beta)\le\epsilon$.
\end{enumerate}
\end{enumerate}
\item Suppose that $\alpha\in a_{\gamma'}$.
\begin{enumerate}
\item If $\beta\in a_\gamma$, then $\alpha>\beta$ which gives a contradiction to $\pair{\alpha,\beta}\in[z]^2$;
\item If $\beta\in r\uplus a_{\gamma'}$, then $\Delta(\alpha,\beta)\le\epsilon$.
\end{enumerate}
\end{enumerate}

Finally, by $z\supseteq (x\symdiff y)=a_\gamma\uplus a_{\gamma'}$,
we have $a_\gamma\times a_{\gamma'}\s [z]^2$, so that case (2)(b)(i)
is indeed feasible.
\end{proof}

As $\Omega=h[m\times\{\delta\}]$, we have $f_\gamma(\iota) \lhd f_{\gamma'}(\iota)$ iff $\iota\in h[m\times\{\delta\}]$.
Let $j<m$ be such that $d(z)=c(a_\gamma(j),a_{\gamma'}(j))$.
Then, for all $\tau\le\delta$:
$f_\gamma(h(j,\tau))\lhd f_{\gamma'}(h(j,\tau))$ iff $h(j,\tau)\in h[m\times\{\delta\}]$ iff $\tau=\delta$.
That is, for all $\tau\le\delta$:
$l_{a_\gamma(j)}(\tau)\lhd l_{a_{\gamma'}(j)}(\tau)$ iff $\tau=\delta$.
Recalling the definition of $c$, we altogether infer that $d(z)=c(a_\gamma(j),a_{\gamma'}(j))=\delta$, as sought.
\end{proof}

\begin{corollary}\label{alphasuccessor} For every successor ordinal $\alpha$:
\begin{enumerate}
\item $\axiom(\beth_\alpha,\omega)$ and $\axiomi(\aleph_\alpha,\cf(\aleph_{\alpha-1}),\omega)$ hold;\footnote{Note that if $\aleph_{\alpha-1}$ is regular, then $\axiom^*(\aleph_\alpha,\cf(\aleph_{\alpha-1}),\omega)$
is equivalent to $\axiom^*(\aleph_\alpha,\aleph_\alpha,\omega)$.}
\item if $\beth_\alpha=\aleph_\alpha$, then $\axiomi(\aleph_\alpha,\aleph_\alpha,\omega)$ holds.\footnote{Note that $\gch$ is equivalent to the assertion that $\beth_\alpha=\aleph_\alpha$ for all ordinals $\alpha$,
and that if $\zfc$ is consistent then so is $\zfc+\gch~+$ every regular uncountable cardinal is of the form $\aleph_\alpha$ for some successor ordinal $\alpha$.}
\end{enumerate}

For every limit ordinal $\alpha$:
\begin{enumerate}
\item[(3)] if $\cf(\alpha)$ is uncountable and admitting a nonreflecting stationary set, then $\axiomi(\aleph_\alpha,\cf(\alpha),\omega)$ holds;
\item[(4)] if $\cf(\alpha)$ is a successor of an infinite cardinal of cofinality $\theta$, then $\axiomi(\aleph_\alpha,\theta,\omega)$ holds.
\end{enumerate}
\end{corollary}
\begin{proof} \begin{enumerate}

\item Suppose that $\alpha=\beta+1$.

Then $\lambda:=\beth_\beta$ is a strong limit cardinal, and hence $2^{<\lambda}=\lambda$.
So, by Theorem~\ref{power}, $\axiom(\cf(2^\lambda),\omega)$ holds.
But, then, by Proposition~\ref{singulars}, $\axiom(2^\lambda,\omega)$ holds.
That is, $\axiom(\beth_{\alpha},\omega)$ holds.

As for the second part of Clause~(1):

$\br$ If $\aleph_\beta$ is a regular cardinal, then by Corollary~\ref{c36}, $\axiomi(\aleph_\alpha,\aleph_\alpha,\omega)$ holds.

$\br$ If $\aleph_\beta$ is a singular cardinal,  then by Fact~\ref{factpr1}(\ref{successor_of_singular}), $\pr_1(\aleph_\alpha,\aleph_\alpha,\cf(\aleph_\beta),\cf(\aleph_\beta))$ holds.
Then, by Lemma~\ref{pr1lemma}, $\axiomi(\aleph_\alpha,\cf(\aleph_\beta),\omega)$ holds.

\item Suppose that $\alpha=\beta+1$ and $\beth_\alpha=\aleph_\alpha$.
Write $\lambda:=\aleph_\beta$.

$\br$ If $\lambda$ is a regular cardinal, then by Corollary~\ref{c36}, $\axiomi(\aleph_\alpha,\aleph_\alpha,\omega)$ holds.

$\br$ If $\lambda$ is a singular cardinal,  then $\pp(\lambda)\le 2^\lambda=2^{\aleph_\beta}\le 2^{\beth_\beta}=\beth_\alpha=\aleph_\alpha=\lambda^+$,
and then by Fact~\ref{factpr1}(\ref{pp}), $\pr_1(\lambda^+,\lambda^+,\lambda^+,\cf(\lambda))$ holds.
So, by Lemma~\ref{pr1lemma}, $\axiomi(\aleph_\alpha,\aleph_\alpha,\omega)$ holds.

\item Suppose that $\cf(\alpha)=\kappa$, where $\kappa$ is an uncountable cardinal admitting a nonreflecting stationary set.
By Fact~\ref{factpr1}(\ref{nonreflecting}), $\pr_1(\kappa,\kappa,\kappa,\omega)$ holds. Then, by Lemma~\ref{pr1lemma}, $\axiomi(\kappa,\kappa,\omega)$ holds.
As $\cf(\aleph_\alpha)=\cf(\alpha)=\kappa$, we infer from Proposition~\ref{singulars} that $\axiomi(\aleph_\alpha,\cf(\alpha),\omega)$ holds.

\item Suppose that $\cf(\alpha)=\mu^+$ for some infinite cardinal $\mu$ of cofinality, say, $\theta$.
Given Clause~(3), we may assume that $\mu$ is singular.
By Fact~\ref{factpr1}(\ref{successor_of_singular}), $\pr_1(\mu^+,\mu^+,\theta,\theta)$ holds.
Then, by Lemma~\ref{pr1lemma}, $\axiomi(\mu^+,\theta,\omega)$ holds.
As $\cf(\aleph_\alpha)=\cf(\alpha)=\mu^+$, we infer from Proposition~\ref{singulars} that $\axiomi(\aleph_\alpha,\theta,\omega)$ holds. \qedhere
\end{enumerate}
\end{proof}

\begin{corollary}\label{bethsuccessor} For every ordinal $\alpha$ such that $\cf(\alpha)$ is a successor cardinal, $\axiom(\beth_\alpha,\omega)$ holds.
\end{corollary}
\begin{proof} By Corollary~\ref{alphasuccessor} and Proposition~\ref{prop32}.
\end{proof}

\begin{remarks}
\begin{enumerate}[i.]
\item The restriction to cofinality of a successor cardinal is necessary, as it follows from Proposition~\ref{prop32}(4)
that $\axiom(\beth_\alpha,\omega)$ fails for any ordinal $\alpha$ satisfying $\alpha\rightarrow[\alpha]^2_2$.
\item It is unknown whether the conclusion $\axiom(\beth_\alpha,\omega)$ may be replaced by the stronger conclusion $\axiomi(\beth_\alpha,\omega,\omega)$.
In fact, already whether $\zfc$ implies $\beth_1\nrightarrow[\beth_1;\beth_1]^2_2$ is a  longstanding open problem.
\end{enumerate}
\end{remarks}

\section{Colourings for sumsets and bounded finite sums}\label{negativehindman}

The main goal of this section is to show that for unboundedly many regular uncountable cardinals $\kappa$, if
$|G|=\kappa$ then $G\nrightarrow[\kappa]_{\kappa}^{\fs_2}$ and even $G\nrightarrow[\kappa]_{\kappa}^{\sumsets}$.
A minor goal is to prove some no-go theorems.

\subsection{Sumsets}
We commence with a lemma that will simplify some reasoning concerning sumsets.

\begin{lemma}\label{reductionsumsets}
Let $G$ be a commutative cancellative semigroup of cardinality $\kappa>\omega$,
let $\theta\leq\kappa$ be an arbitrary cardinal, and let $c:G\longrightarrow\theta$ be some colouring.
Then for each $\delta<\theta$, the following two conditions are equivalent:
\begin{enumerate}
\item For every $X,Y\subseteq G$ with $|X|=|Y|=\kappa$, we have
$\delta\in c[X+Y]$ (that is, there are $x\in X$ and $y\in Y$ such that
$c(x+y)=\delta$).
\item For every integer $n\ge2$ and $\kappa$-sized sets
$X_1,\ldots,X_n\subseteq G$, we have $\delta\in c[X_1+\cdots+X_n]$ (that is,
there are $x_1\in X_1,\ldots,x_n\in X_n$ such that
$c(x_1+\cdots+x_n)=\delta$).
\end{enumerate}
\end{lemma}

\begin{proof}We focus on the nontrivial implication $1\implies2$.

Proceed by induction on $n\in\{2,3,\ldots\}$. Suppose the statement
holds true for a given $n$, and suppose we are given $\kappa$-sized sets
$X_1,\ldots,X_n,X_{n+1}$. First, notice that
$X_1+\cdots+X_n$ has cardinality $\kappa$ (since
the elements $x_1+\cdots+x_{n-1}+y$, where the $x_i\in X_i$ are
fixed, and $y$ ranges over $X_n$, are all distinct
by cancellativity). Thus, by our assumption we can
find $x_1+\cdots+x_n\in X_1+\cdots+X_n$ and $x_{n+1}\in X_{n+1}$
such that $c(x_1+\cdots+x_n+x_{n+1})=c((x_1+\cdots+x_n)+x_{n+1})=\delta$.
\end{proof}

\begin{theorem}\label{maintheorem2} Suppose that $G$ is commutative cancellative semigroup of cardinality, say, $\kappa$.

If $\axiomi(\kappa,\theta,\omega)$ holds, then so does $G\nrightarrow[\kappa]_{\theta}^{\sumsets}$.
\end{theorem}
\begin{proof}
Let $d:[\kappa]^{<\omega}\longrightarrow\theta$ be a witness to $\axiomi(\kappa,\theta,\omega)$.
Using Lemma~\ref{embedding}, embed $G$ into a direct sum
$\bigoplus_{\alpha<\kappa}G_\alpha$, with each $G_\alpha$
a countable abelian group.
Then, define $c:G\rightarrow\theta$ by stipulating:
$$c(x):=d(\supp(x)).$$

While the axiom $\axiomi(\kappa,\theta,\omega)$ allows to handle any finite number of $\kappa$-sized families,
we shall take advantage of Lemma~\ref{reductionsumsets} that reduces the algebraic problem into looking at sumsets of $2$ sets.

Thus, let $X$ and $Y$ be two $\kappa$-sized subsets of $G$,
and let $\delta<\theta$ be arbitrary.
Since each $x\in G$ has a finite support, and each
of the $G_\alpha$ are countable, there are only
countably many elements of $G$ with a
given support. Therefore,
both $\mathcal X:=\{\supp(x)\mid x\in X\}$ and
$\mathcal Y:=\{\supp(y)\mid y\in Y\}$ are $\kappa$-sized
subfamilies of $[\kappa]^{<\omega}$.
Now, as $d$ witnesses $\axiomi(\kappa,\theta,\omega)$, we
may  pick $\pair{x,y}\in X\times Y$
such that $d(z)=\delta$
whenever
\begin{equation*}
\supp(x)\symdiff\supp(y)\s z\s \supp(x)\cup\supp(y),
\end{equation*}
in particular $z:=\supp(x+y)$ satisfies the above equation
by Proposition~\ref{supportofsums},
and hence $c(x+y)=d(z)=\delta$.
\end{proof}

\begin{corollary}\label{maincorollary}
Let $G$ be any commutative cancellative semigroup of cardinality, say, $\lambda$.

If $\kappa:=\cf(\lambda)$ is an uncountable cardinal satisfying at least one of the following conditions:
\begin{enumerate}
\item $\square(\kappa)$ holds;
\item $\kappa$ admits a non-reflecting stationary set (e.g., $\kappa=\mu^+$ for $\mu$ regular);
\item $\kappa=\mu^+$, $\mu$ is singular, and $\pp(\mu)=\mu^+$ (e.g., $\mu^{\cf(\mu)}=\mu^+$),
\end{enumerate}
then $G\nrightarrow[\lambda]_{\kappa}^{\sumsets}$ holds.
\end{corollary}
\begin{proof} By Lemma~\ref{pr1lemma}, Fact~\ref{factpr1} and Theorem~\ref{lowertrace},
any of the above hypotheses imply that $\axiomi(\kappa,\kappa,\omega)$ holds.
Then by Proposition~\ref{singulars}, $\axiomi(\lambda,\kappa,\omega)$ holds, as well.
Now, appeal to Theorem~\ref{maintheorem2}.
\end{proof}

\begin{corollary}\label{cor44} It is consistent with $\zfc$
that for every infinite commutative cancellative semigroup $G$, letting $\kappa:=|G|$,
$G\nrightarrow[\kappa]_{{\cf(\kappa)}}^{\sumsets}$ holds iff $\cf(\kappa)>\omega$.
\end{corollary}
\begin{proof} If there exists a weakly compact cardinal in G\"odel's constructible universe $L$,
then let $\mu$ denote the least such one and work in $L_\mu$. Otherwise, work in $L$.
In both cases, we end up with a model of $\zfc$ in which $\square(\kappa)$ holds for every regular uncountable cardinal $\kappa$ \cite[Theorem 6.1]{MR0309729},
and in which every singular cardinal is a strong limit \cite{MR0002514}.
Now, there are three cases to consider:
\begin{itemize}
\item[$\br$] If $\cf(\kappa)>\omega$, then $\square(\cf(\kappa))$ holds and then $G\nrightarrow[\kappa]_{{\cf(\kappa)}}^{\sumsets}$ holds as a consequence of Corollary~\ref{maincorollary}.
\item[$\br$] If $\kappa>\cf(\kappa)=\omega$ and $G\nrightarrow[\kappa]_{\cf(\kappa)}^{\sumsets}$ holds,
then so does $G\nrightarrow[\kappa]_{\omega}^{\fs_2}$.
But then by Proposition~\ref{prop_FSn} below, $\kappa\nrightarrow[\kappa]_{\omega}^{2}$ holds,
contradicting Theorem~54.1 of \cite{MR795592} and the fact that $\kappa$ is a strong limit.
\item[$\br$] If $\kappa=\omega$ and $G\nrightarrow[\kappa]_{\cf(\kappa)}^{\sumsets}$ holds,
then so does $G\nrightarrow[\omega]_{2}^{\fs}$, contradicting Hindman's theorem. \qedhere
\end{itemize}
\end{proof}

We conclude this subsection with an analogue of Corollary~\ref{c29} in the context of sumsets:

\begin{corollary} For every infinite cardinal $\lambda$, and every cardinal $\theta$, the following are equivalent:
\begin{enumerate}
\item $\lambda^+\nrightarrow[\lambda^+]^2_\theta$ holds;
\item $G\nrightarrow[\lambda^+]^{\sumsets}_\theta$ holds for every commutative cancellative semigroup $G$ of cardinality $\lambda^+$;
\item $G\nrightarrow[\lambda^+]^{\sumsets}_\theta$ holds for some commutative cancellative semigroup $G$ of cardinality $\lambda^+$.
\end{enumerate}
\end{corollary}
\begin{proof} Let $\lambda$ and $\theta$ be as above. If $\lambda$ is regular,
then all of the three clauses hold as a consequence of Corollary~\ref{c36} and Theorem~\ref{maintheorem2}.

Next, as the implication $(2)\implies (3)$ is trivial, and the implication $(3)\implies(1)$ follows immediately from Proposition~\ref{prop_FSn} below,
let us suppose that $\lambda$ is a singular cardinal for which $\lambda^+\nrightarrow[\lambda^+]^2_\theta$ holds.

Then, by \cite[Theorem 1]{rinot13}, $\pr_1(\lambda^+,\lambda^+,\theta,\cf(\lambda))$ holds.
Then, by Lemma~\ref{pr1lemma} and Theorem~\ref{maintheorem2},
$G\nrightarrow[\lambda^+]^{\sumsets}_\theta$ holds for every commutative cancellative semigroup $G$ of cardinality $\lambda^+$.
\end{proof}

\subsection{Finite sums}

An immediate corollary to Theorem~\ref{maintheorem2} reads as follows.

\begin{corollary} Suppose that $G$ is commutative cancellative semigroup of cardinality, say, $\kappa$.

If $\axiomi(\kappa,\theta,\omega)$ holds, then so does $G\nrightarrow[\kappa]_{\theta}^{\fs_n}$ for all integers $n\ge 2$.
\end{corollary}

Our next goal is to derive statements about $\fs_n$ from the weaker principle $\axiom(\kappa,\theta)$. We first deal with the case where $n=2$.

\begin{theorem}\label{maintheorem} Suppose that $G$ is commutative cancellative semigroup of cardinality, say, $\kappa$.

If $\axiom(\kappa,\theta)$ holds, then so does $G\nrightarrow[\kappa]_{\theta}^{\fs_2}$.\footnote{Note that when $G$ is the abelian group $([\kappa]^{<\omega},\symdiff)$,
then a colouring witnessing $G\nrightarrow[\kappa]^{\fs_2}_\theta$ is \emph{almost} a witness to $\axiom(\kappa,\theta)$.}
\end{theorem}
\begin{proof}
Let $d:[\kappa]^{<\omega}\longrightarrow\theta$ be a witness to $\axiom(\kappa,\theta)$.
Using Lemma~\ref{embedding}, embed $G$ into a direct sum
$\bigoplus_{\alpha<\kappa}G_\alpha$, with each $G_\alpha$
a countable abelian group.
Then, define $c:G\rightarrow\theta$ by stipulating:
$$c(x):=d(\supp(x)).$$

Let $X\in[G]^\kappa$ be arbitrary.
Since each $x\in G$ has a finite support, and
there are only countably many elements of $G$ with a given support,
we have that $\mathcal X:=\{\supp(x)\mid x\in X\}$ is a subfamily of $[\kappa]^{<\omega}$ of size $\kappa$.
Let $\delta<\theta$ be arbitrary. As $d$ witness $\axiom(\kappa,\theta)$, we may now pick two distinct $x,y\in X$ such that $d(z)=\delta$
whenever $\supp(x)\symdiff\supp(y)\s z\s \supp(x)\cup \supp(y)$.
In particular, by Proposition~\ref{supportofsums}, we have that
$\supp(x+y)$ is such a $z$, and therefore $c(x+y)=d(\supp(x+y))=\delta$.
\end{proof}

We would next like to obtain the corresponding result for $\fs_n$,
with $n>2$. This will, however, not be very hard under the right
circumstances, as the following
lemma shows. The idea for the proof of this lemma is adapted
from~\cite{komjath}. Recall that, given an $n\in\mathbb N$,
an abelian group $G$ is said to be \emph{$n$-divisible} if
for every $x\in G$ there exists a $z\in G$ such that $nz=x$
(thus, being divisible is the same as being $n$-divisible
for every $n\in\mathbb N$).

\begin{lemma}\label{transfertwotok}
Let $n\in\mathbb N$, and let $G$ be an $n$-divisible abelian group.
For every $\lambda,\theta$, if $c$ is a colouring witnessing
$G\nrightarrow[\lambda]_\theta^{\fs_n}$, then $c$ witnesses
$G\nrightarrow[\lambda]_{\theta}^{\fs_{n+1}}$, as well.
\end{lemma}
\begin{proof}
Suppose that $G$ is an $n$-divisible abelian group such that
$G\nrightarrow[\lambda]_\theta^{\fs_n}$ holds,
as witnessed by a colouring $c:G\longrightarrow\theta$.
To see that it is also the case that for all $X\subseteq G$ with $|X|=\lambda$,
$c[\fs_{n+1}(X)]=\theta$, grab an arbitrary $X\subseteq G$ with $|X|=\lambda$.
Pick an element $x\in X$, and use $n$-divisibility to obtain an element
$z\in G$ such that $nz=x$. Now, let
\begin{equation*}
Y:=\left\{y+z\mid y\in X\setminus\{x\}\right\}.
\end{equation*}
Then $Y$ is a subset of $G$ of cardinality $\lambda$,
so that for each colour $\delta<\theta$ we can find $n$ distinct elements
$y_1+z,\ldots,y_n+z\in Y$
such that the sum
\begin{equation*}
(y_1+z)+\cdots+(y_n+z)=y_1+\cdots+y_n+nz=y_1+\cdots+y_n+x
\end{equation*}
is an element of $\fs_{n+1}(X)$ that receives colour $\delta$.
\end{proof}

The preceding lemma yields a fairly
general result concerning $\fs_n$, where the main piece
of information seems to be the cardinality of the commutative
cancellative semigroup $G$, rather than the semigroup itself.

\begin{corollary}\label{transfertwotow}
Suppose that $\theta\le\kappa$ are cardinals such that $G\nrightarrow[\lambda]_\theta^{\fs_2}$
holds for all commutative cancellative semigroups of cardinality $\kappa$.

Then, for every commutative cancellative semigroup $G$ of cardinality $\kappa$,
there exists a colouring $c$ that simultaneously witnesses $G\nrightarrow[\lambda]_\theta^{\fs_n}$ for all integers $n\ge2$.
\end{corollary}
\begin{proof}
Let $G$ be any commutative cancellative semigroup $G$ with $|G|=\kappa$.
By Lemma~\ref{embedding}, we may embed $G$ into a direct sum
$G'=\bigoplus_{\alpha<\kappa}G_\alpha$, where each $G_\alpha$ is
countable and divisible. This implies that $G'$ is divisible,
and $|G'|=\kappa$. Thus, by our hypothesis, we may take a colouring
$d:G'\longrightarrow\theta$ witnessing $G'\nrightarrow[\lambda]_\theta^{\fs_2}$.
Since $G'$ is divisible, we can use Lemma~\ref{transfertwotok},
to inductively prove, for every integer $n\ge2$, that $d$ witnesses
$G'\nrightarrow[\lambda]_\theta^{\fs_n}$.
Therefore $c:=d\upharpoonright G$ witnesses the
statement $G\nrightarrow[\lambda]_\theta^{\fs_n}$ for all integers $n\ge2$.
\end{proof}

We now move to proving no-go propositions. These will be obtained using the following simple proxy:

\begin{prop}\label{prop_FSn} If $G$ is a commutative semigroup satisfying $G\nrightarrow[\lambda]^{\fs_n}_\theta$,
then $\kappa\nrightarrow[\lambda]^n_\theta$ holds, for $\kappa:=|G|$.
\end{prop}
\begin{proof} Fix an injective enumeration $\{x_\alpha\mid\alpha<\kappa\}$ of a commutative semigroup $G$,
along with a colouring $c:G\longrightarrow\theta$ witnessing $G\nrightarrow[\lambda]^{\fs_n}_\theta$.
Define a colouring $d:[\kappa]^n\longrightarrow\theta$ by stipulating:
$$d(\alpha_1,\ldots,\alpha_n)=c(x_{\alpha_1}+\cdots+x_{\alpha_n}).$$

Now, given $Y\in [\kappa]^\lambda$, we have that $\{ x_\alpha\mid \alpha\in Y\}$
is a $\lambda$-sized subset of $G$ and hence for every colour $\delta<\theta$,
we may find ${\alpha_1},\ldots,{\alpha_n}$ in $Y$ such that $c(x_{\alpha_1}+\cdots+x_{\alpha_n})=\delta$,
so that $d(\alpha_1,\ldots,\alpha_n)=\delta$.
\end{proof}

\begin{corollary} There exists an uncountable cardinal $\kappa$ such that for every commutative semigroup $G$ of cardinality $\kappa$,
$G\nrightarrow[\kappa]^{\fs_n}_\omega$ fails for all $n<\omega$.

In particular, there exists an uncountable abelian group $G$ for which $G\nrightarrow[\omega_1]^{\fs_n}_\omega$ fails for all $n<\omega$.
\end{corollary}
\begin{proof} The statement holds true for $\kappa:=\beth_\omega$,
as otherwise, by Proposition~\ref{prop_FSn},
there exists some $n<\omega$ for which $\beth_\omega\nrightarrow[\beth_\omega]^n_\omega$ holds,
contradicting Theorem~54.1 of \cite{MR795592}.

In particular, by picking any abelian group $G$ of cardinality $\beth_\omega$,
we infer that $G\nrightarrow[\omega_1]^{\fs_n}_\omega$ fails for all $n<\omega$.
\end{proof}

\begin{corollary}\label{c410} Suppose that $\kappa$ is a weakly compact cardinal.

Then in the generic extension for adding $\kappa$ many Cohen reals,
for every commutative semigroup $G$ of size continuum, $G\nrightarrow[\mathfrak c]^{\fs_n}_{\omega_1}$ fails for every integer $n\ge2$.
\end{corollary}
\begin{proof}
Let us remind the reader that as $\kappa$ is weakly compact,
$\kappa$ is strongly inaccessible and satisfies that for every $\theta<\kappa$, every positive integer $n$, and every colouring $d:[\kappa]^n\rightarrow\theta$,
there exists some $H\in[\kappa]^{\kappa}$ such that $d\restriction[H]^n$ is constant.

Let $\mathbb P$ be the notion of forcing for adding $\kappa$ many Cohen reals.
By Proposition~\ref{prop_FSn}, it suffices to show that in $V^{\mathbb P}$, we have $\kappa\rightarrow[\kappa]^n_{\omega_1}$ for every integer $n\ge2$.

Let $\mathring{c}$ be a $\mathbb P$-name for an arbitrary colouring $c:[\kappa]^n\rightarrow\omega_1$ in $V^{\mathbb P}$.
Working in $V$, define a colouring $d:[\kappa]^n\rightarrow[\omega_1]^{\le\omega}$
by stipulating:
$$d(\alpha_1,\ldots,\alpha_n):=\{ \delta<\omega_1\mid \exists p\in\mathbb P[p\Vdash``\mathring{c
}(\check{\alpha_1},\ldots,\check{\alpha_n})=\check{\delta}"]\}.$$

Since $\mathbb P$ is \textit{ccc}, the range of $d$ indeed consists of \emph{countable} subsets of $\omega_1$.
As $\kappa$ is weakly compact, $\left|[\omega_1]^{\le\omega}\right|<\kappa$ and
we may pick some $H\in[\kappa]^{\kappa}$ such that $d``[H]^n$ is a singleton, say, $\{A\}$.
Evidently,  $$\Vdash ``\mathring{c}[[\check{H}]^{\check{n}}]\s \check A".$$

As $A$ is countable and $\mathbb P$ does not collapse cardinals, $\mathbb P$ forces that $\mathring{c}``[\check{H}]^{\check{n}}$ omits at least one colour.
\end{proof}

\section{Some corollaries concerning the real line}\label{therealline}
Hindman, Leader and Strauss~\cite[Theorem~3.2]{hindman-leader-strauss} proved
that $\mathbb R\nrightarrow[\mathfrak c]_2^{\fs_n}$ holds for every integer $n\ge2$.
It turns out that it is possible to  increase the number of colours from $2$ to $\omega$:
\begin{corollary} $\mathbb R\nrightarrow[\mathfrak c]_\omega^{\fs_n}$ holds for every integer $n\ge2$.
\end{corollary}
\begin{proof} As $|\mathbb R|=\mathfrak c=2^{\aleph_0}=2^{\beth_0}=\beth_1$, we infer from Corollary~\ref{bethsuccessor} that $\axiom(\mathfrak c,\omega)$ holds. Now appeal to Theorem~\ref{maintheorem2} and Corollary~\ref{transfertwotow}.
\end{proof}

On the grounds of $\zfc$ alone,  it is impossible to increase the number of colours to $\omega_1$:
\begin{corollary}\label{justify}
If there exists a weakly compact cardinal, then there exists a model of $\zfc$
in which $\mathbb R\nrightarrow[\mathfrak c]^{\fs_n}_{\omega_1}$ fails for every integer $n\ge2$.
Furthermore, in this model, $\mathfrak c$ is an inaccessible cardinal which is weakly compact in G\"odel's constructible universe.
\end{corollary}
\begin{proof}  Suppose that $\kappa$ is a weakly compact cardinal.
Work in G\"odel's constructible universe, $L$. Then $\kappa$ is still a weakly compact cardinal (see \cite[Theorem 17.22]{MR1940513}),
and if $\mathbb P$ denotes the notion of forcing for adding $\kappa$ many Cohen reals,
then by Corollary~\ref{c410}, the forcing extension $L^{\mathbb P}$
is a model in which $\mathfrak c$ is an inaccessible cardinal that is weakly compact in G\"odel's constructible universe,
and $\mathbb R\nrightarrow[\mathfrak c]^{\fs_n}_{\omega_1}$ fails for every integer $n\ge2$.
\end{proof}

However, assuming an anti-large cardinal hypothesis, the number of colours may be increased:

\begin{corollary}\label{c51} Each of the following imply that $\mathbb R\nrightarrow[\mathfrak c]^\sumsets_{\mathfrak c}$ holds:
\begin{enumerate}
\item $\mathfrak c=\mathfrak b$ (e.g., Martin's Axiom holds);
\item $\mathfrak c$ is a successor of a regular cardinal (e.g., $\ch$ holds);
\item $\mathfrak c$ is a successor of a singular cardinal of countable cofinality;
\item $\mathfrak c$ is a regular cardinal that is not weakly compact in G\"odel's constructible universe.
\end{enumerate}
\end{corollary}
\begin{proof} \begin{enumerate}
\item  By Lemma~\ref{pr1lemma}, Fact~\ref{factpr1}(1), and Theorem~\ref{maintheorem2}.

\item By Corollary~\ref{maincorollary}(2).

\item If $\mathfrak c=\lambda^+$ and $\cf(\lambda)=\omega$, then $\lambda^{\cf(\lambda)}=\mathfrak c=\lambda^+$.
Now, appeal to Corollary~\ref{maincorollary}(3).

\item By \cite{MR908147}, if $\kappa$ is a regular uncountable cardinal which is not a weakly compact cardinal in G\"odel's constructible universe,
then $\square(\kappa)$ holds. Now, appeal to Corollary~\ref{maincorollary}(1). \qedhere
\end{enumerate}
\end{proof}

By a theorem of Milliken \cite[Theorem~9]{MR0505558}, $\ch$ entails $\mathbb R\nrightarrow[\mathfrak c]_{\omega_1}^{\fs_2}$.
We now derive the same conclusion (and even with superscript $\sumsets$) from an (again, optimal) anti-large cardinal hypothesis:

\begin{corollary}\label{c52} Each of the following imply that $\mathbb R\nrightarrow[\mathfrak c]^\sumsets_{\omega_1}$ holds:
\begin{enumerate}
\item $\mathfrak c$ is a successor cardinal  (e.g., $\ch$ holds);
\item $\cf(\mathfrak c)$ is a successor of a cardinal of uncountable cofinality;
\item $\cf(\mathfrak c)$ is not weakly compact in G\"odel's constructible universe.
\end{enumerate}
\end{corollary}
\begin{proof} \begin{enumerate}
\item By Corollary~\ref{c51}, we may assume that $\mathfrak c$ is a successor of a singular cardinal of uncountable cofinality $\theta$.
Then, By Fact~\ref{factpr1}(7), Lemma~\ref{pr1lemma}, and Theorem~\ref{maintheorem2},
$\mathbb R\nrightarrow[\mathfrak c]^\sumsets_{\theta}$ holds.
In particular, $\mathbb R\nrightarrow[\mathfrak c]^\sumsets_{\omega_1}$ holds.
\item If $\cf(\mathfrak c)$ is a successor of a regular cardinal,
then by Corollary~\ref{c36}, $\axiomi(\cf(\mathfrak c),\cf(\mathfrak c),\omega)$ holds.
If $\cf(\mathfrak c)$ is a successor of a singular cardinal of cofinality $\theta$,
then, by Fact~\ref{factpr1}(7) and Lemma~\ref{pr1lemma}, $\axiomi(\cf(\mathfrak c),\theta,\omega)$ holds.
Altogether, if $\cf(\mathfrak c)$ is a successor of a cardinal of uncountable cofinality,
$\axiomi(\cf(\mathfrak c),\omega_1,\omega)$ holds.
So, by Proposition~\ref{singulars}, $\axiomi(\mathfrak c,\omega_1,\omega)$ holds.
By Theorem~\ref{maintheorem2}, then, $\mathbb R\nrightarrow[\mathfrak c]^\sumsets_{\omega_1}$ holds.
\item Let $\kappa:=\cf(\mathfrak c)$. By K\"onig's lemma, $\kappa$ is uncountable. By \cite{MR908147}, if $\kappa$ is a regular uncountable cardinal which is not a weakly compact cardinal in G\"odel's constructible universe,
then $\square(\kappa)$ holds. Thus, by Corollary~\ref{maincorollary}(1), $\mathbb R\nrightarrow[\mathfrak c]^\sumsets_{\kappa}$ holds.
In particular, $\mathbb R\nrightarrow[\mathfrak c]^\sumsets_{\omega_1}$ holds. \qedhere
\end{enumerate}
\end{proof}

As for $\fs$ sets, we have the following optimal results:

\begin{corollary} The following are equivalent:
\begin{itemize}
\item $\mathbb R\nrightarrow[\mathfrak c]^\fs_{\mathfrak c}$ holds;
\item $\mathfrak c$ is not a J\'onsson cardinal.
\end{itemize}
\end{corollary}
\begin{proof} Appeal to Corollary~\ref{c29} with $\lambda=\kappa=\theta=\mathfrak c$.
\end{proof}

\begin{corollary} The following are equivalent:
\begin{itemize}
\item $\mathbb R\nrightarrow[\omega_1]^\fs_{\omega_1}$ holds;
\item $(\mathfrak c,\omega_1)\twoheadrightarrow(\omega_1,\omega)$ fails.
\end{itemize}
\end{corollary}
\begin{proof} Appeal to Corollary~\ref{chang} with $\kappa=\mathfrak c$ and $\theta=\omega$.
\end{proof}

In particular, if there exists a Kurepa tree with $\mathfrak c$ many branches, then $\mathbb R\nrightarrow[\omega_1]^\fs_{\omega_1}$ holds.

\section*{Acknowledgement}
The first author would like to thank Andreas Blass for many fruitful discussions.

\end{document}